\documentclass{amsart} 

\usepackage{amsrefs, hyperref}
\numberwithin{equation}{section}

\newtheorem{thm}{Theorem}[section]
\newtheorem{theorem}[thm]{Theorem}

\newtheorem{cor}[thm]{Corollary}
\newtheorem{prop}[thm]{Proposition}

\theoremstyle{definition}

\newtheorem{remark}[thm]{Remark}

\newtheorem{defn}[thm]{Definition}

\newtheorem{defn-thm}[thm]{Definition-Theorem}

\usepackage{graphicx}
\usepackage{mathtools}
\DeclarePairedDelimiter{\ceil}{\lceil}{\rceil}
\DeclarePairedDelimiter{\floor}{\lfloor}{\rfloor}

\usepackage{amsmath,amsfonts,amsthm,amssymb,graphics,color,datetime}

\usepackage{tikz}
\usetikzlibrary{angles,quotes,calc}
\newcommand{\RN}[1]{%
  \textup{\uppercase\expandafter{\romannumeral#1}}%
}

\usepackage{tikz}
\usetikzlibrary{decorations.markings}
\usetikzlibrary{angles,quotes,calc}
\usetikzlibrary{calc}
 \usepackage{subfigure}

\def\beq{\begin{eqnarray}}
\def\eeq{\end{eqnarray}}
\newcommand{\nn}{\nonumber}
\newcommand{\kmax}{k_{max}}
\newcommand{\Rea}{\operatorname{Re}} 

\numberwithin{equation}{section}

        \definecolor{pink}{rgb}{1,0,1}
        \definecolor{purple}{rgb}{0.4,0.2,1}

\newcommand{\bM}{{\bf M}} 
\newcommand{\ovM}{\overline{M}} 
\newcommand{\Ker}{\operatorname{Ker}}

\newcommand{\Scal}{\operatorname{Scal}}

\newcommand{\pa}{\partial}
\newcommand{\eps}{\varepsilon}

\newcommand{\tr}{{\rm Tr}}
\newcommand{\Tr}{{\rm Tr}}

\newcommand{\cN}{{\mathcal{N}}}
\newcommand{\cM}{{\mathcal{M}}}

\newcommand{\N}{\mathbb{N}}
\newcommand{\R}{\mathbb{R}}

\newcommand{\C}{\mathbb{C}}
\newcommand{\Z}{\mathbb{Z}}

\newcommand{\bS}{\mathbb{S}}
\newcommand{\cS}{\mathcal{S}}
\newcommand{\cC}{\mathcal{C}}
\newcommand{\cL}{\mathcal{L}}

\newcommand{\wt}[1]{\widetilde{#1}}

\begin{document}

\title[Polyakov formulas for conical singularities in two dimensions]{Polyakov formulas for conical singularities in two dimensions}  
\author[C.-L.~Aldana, K.~Kirsten, and J.~Rowlett]{Clara L. Aldana, Klaus Kirsten, and Julie Rowlett}

\begin{abstract}
We investigate the zeta-regularized determinant and its variation in the presence of conical singularities, boundaries, and corners.  For surfaces with isolated conical singularities which may also have one or more smooth boundary components, we demonstrate both a variational Polyakov formula as well as an integrated Polyakov formula for the conformal variation of the Riemannian metric with conformal factors which are smooth up to all singular points and boundary components.  We demonstrate the analogous result for curvilinear polygonal domains in surfaces.  We then specialize to finite circular sectors and cones and via two independent methods obtain variational Polyakov formulas for the dependence of the determinant on the opening angle.  Notably, this requires the conformal factor to be logarithmically singular at the vertex.  We further obtain explicit formulas for the determinant for finite circular sectors and cones.   \end{abstract}

\maketitle

\section{Introduction} \label{s:intro}
In physics, manifolds with conical singularities are of great importance. In particular in the context of quantum field theory in curved
spacetime the influence of such singularities has been analyzed in detail \cite{dowk77-10-115}. Instances where these singularities occur
are cosmic strings \cite{vile85-121-263,hell86-33-1918,frol87-35-3779,dowk87-36-3095}, where the cone angle is related to the string tension parameter, and static spacetimes with bifurcate Killing horizons, where the Euclideanized version, considered in finite temperature field theory, has topology $C_\alpha \times S^2$, and the conical angle $\alpha$ is associated to the inverse temperature
\cite{furs94-49-987,furs95-12-393,call94-333-55,suss94-50-2700,dowk94-11-L55,barv95-51-1741}. Renormalization in these theories
necessitates the heat kernel coefficients for manifolds with conical singularities \cite{furs94-334-53,cogn94-49-1029,dowk94-11-L137}.

These coefficients also build the foundation to understand how zeta regularized determinants transform under conformal transformations. In two dimensions, for smooth surfaces with smooth boundaries, this relation has been known for quite some time \cite{alv,lusc80-173-365,poly81-103-207}. In four dimensions, this has been developed in \cite{bran94-344-479} (see also \cite{dowk78-11-895}) for the case of Laplace-type operators on smooth Riemannian manifolds with smooth boundaries. This found applications in the context of effective action in quantum field theory; see
\cite{blau88-209-209,blau89-4-1467,buch86-44-534,dett92-377-252,dowk89-327-267,wipf95-443-201,gusy87-46-1097}. The relation is also essential  to prove certain extremal properties of determinants \cite{bran92-149-241,osgo88-80-212,AAR}. In two dimensions, this so-called Polyakov formula was generalized to the case of piecewise smooth boundary in \cite{dowk94-11-557}, and it has been used to compute functional determinants on different regions of the plane and sphere \cite{dowk94-11-557,dowk95-12-1363}.

\subsection{Geometric setting} \label{ss:geometric_setting}
We consider here compact surfaces with boundary and with finitely many isolated conical singularities, as well as curvilinear polygonal domains contained in larger, ambient, smooth surfaces. Whereas many references assume that conical singularities are \em exact, \em we consider a more general conic singularity of the type given in \cite[Definition 1.1]{mooers}.
Here we specify this definition to two dimensions and adapt it to our needs.

\begin{defn}\label{def:conicmet} Let $\bM$ be a compact $2$-dimensional topological manifold with boundary $\pa \bM$.  Assume that there is a finite set of interior points $\wp := \{p_1, \ldots, p_m\}$ such that $M := {\bM} \setminus \left(\wp \cup \pa {\bM} \right)$ is a smooth, open, manifold of dimension two endowed a with Riemannian metric $g$. The boundary $\pa {\bM}$ is smooth, and the metric $g$ is smooth up to the boundary. Moreover, assume that each $p_i \in \wp$ has a neighborhood of the form
\[ \cN_i \cong [0, \eps_i]_r \times \bS^1, \textrm{ for some } \eps_i > 0. \]
On this neighborhood the Riemannian metric
\[ \left . g\right|_{\cN_i} = dr^2 + r^2 \omega_i(r), \quad r \in [0, \eps_i) \]
where $\omega_i(r)$ is a smooth family of metrics on $\bS^1$ for $r \in [0, \eps_i)$.  The points in $\wp$ are known as \em cone points, conical points, or conical singularities.  \em  The angle at a cone point $p_i$ is defined to be
\[ \gamma_i := \int_{\bS^1} ds_i, \]
where $ds_i = ds_{\omega_i(0)}$, denotes the volume form associated to the metric $\omega_i(0)$.  The angle is assumed to be contained in $(0, 2\pi)$. If $\bf M$ satisfies all these conditions, we call $({\bM}, g)$ a surface with conical singularities and smooth boundary.
\end{defn}

Notice that whereas $\ovM = \bM$, $\pa {\bM} \neq \pa M = \wp \cup \pa \bM$. In addition, the definition allows the metric on the \em link \em of the cone, $\bS^1$, to vary as one approaches the conical singularity.  Some authors may call our definition above a surface with `generalized conical singularities.'  A more rigid definition requires each $\omega_i (r) \equiv \omega_i (0)$ to be a single fixed metric on the link; this is known as a surface with `exact conical singularities.' \\

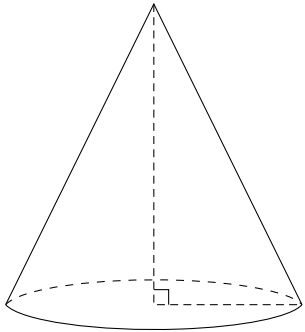
\begin{figure}
  \begin{tikzpicture}
    \draw[dashed] (0,0) arc (170:10:2cm and 0.4cm)coordinate[pos=0] (a);
    \draw (0,0) arc (-170:-10:2cm and 0.4cm)coordinate (b);
    \draw[densely dashed] ([yshift=4cm]$(a)!0.5!(b)$) -- node[right,font=\footnotesize] {$$}coordinate[pos=0.95] (aa)($(a)!0.5!(b)$)
                            -- node[above,font=\footnotesize] {}coordinate[pos=0.1] (bb) (b);
    \draw (aa) -| (bb);
    \draw (a) -- ([yshift=4cm]$(a)!0.5!(b)$) -- (b);

  \end{tikzpicture}
\caption{This is an exact cone with conical singularity at $r=0$.   The $r$ coordinate gives the length along the cone's surface.  The $\theta$ coordinate corresponds to arc-length around the circular edge of the cone.}
  \label{fig:cone}
\end{figure}

We will also consider finite circular sectors and more generally, curvilinear po\-ly\-go\-nal domains in surfaces as well as in the plane.  A finite circular sector in the plane is a set of the form
\[ S_{R,\gamma} := \{ (r, \theta): 0 \leq r \leq R, \quad 0 \leq \theta \leq \gamma \} \subset \R^2. \]
Here we are using standard polar coordinates $(r, \theta)$, and we equip the sector with the Euclidean metric, which in these coordinates is
\[ g = dr^2 + r^2 d\theta^2. \]
Consequently, a finite circular sector is a cone with link $[0, \gamma]$ in the sense that we may identify the finite circular sector with the compact metric space $[0, R] \times [0, \gamma]$ equipped with the metric $g$.  This metric is smooth on the interior and has three singularities, one of which is a conical singularity at $r=0$.  The angle at this singularity is $\gamma$. The other two singularities have angles $\pi/2$ and are examples of conical singularities that are not exact conical singularities.
 A finite circular sector does not fit in Definition \ref{def:conicmet} above because its conical points lie at the boundary, but it is an example of a curvilinear polygonal domain in the plane.  More generally, we define curvilinear polygonal domains in surfaces as in \cite[Definition 1.3]{nrs}.

\begin{defn}\label{def:curvpoly} We say that $\Omega$ is a \em curvilinear polygonal domain \em if it is a subdomain of a smooth, two dimensional Riemannian manifold $(M,g)$ with piecewise smooth boundary and a vertex at each non-smooth point of $\partial\Omega$. A \em vertex \em is a point $p$ on the boundary of $\Omega$ at which the following are satisfied.
\begin{enumerate}
\item The boundary in a neighborhood of $p$ is defined by a continuous curve $\gamma(t): (-a, a) \to M$ for $a > 0$ with $\gamma(0) = p$.  We require that $\gamma$ is smooth {on $(-a,0]$ and $[0,a)$}, with $||\dot \gamma(t)|| =1$ for all $t \in (-a, a)\setminus \{0\}$, and  such that
\[\lim_{t \uparrow 0} \dot \gamma (t) = v_1, \quad \lim_{t \downarrow 0} \dot \gamma  (t) = v_2,\]
for some vectors $v_1,v_2\in T_{p}M$, with $- v_1 \neq v_2$. 
\item The \em interior angle \em at the point $p$ is the \em interior angle \em at that corner, which is the angle between the vectors $-v_1$ and $v_2$.
\end{enumerate}
Note that requiring $-v_1$ and $v_2$ to be distinct means that the interior angle will be an element of $(0,2\pi)$, which rules out inward and outward pointing cusps. An angle of $\pi$, corresponding to a phantom vertex, is allowed.
\end{defn}

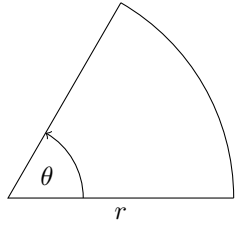
\begin{figure}  \begin{tikzpicture} \coordinate (O) at (0,0); \draw    (3cm,0) coordinate (xcoord) --    node[midway,below] {$r$} (O) --    (60:3cm) coordinate (slcoord)   pic [draw,->,angle radius=1cm,"$\theta$"] {angle = xcoord--O--slcoord}; \draw pic[draw, -, angle radius =3cm]{angle=xcoord--O--slcoord};    \end{tikzpicture} \caption{Above is a finite circular sector.  This can be viewed as a cone where the link of the cone is an interval.  This is an example of a domain with corners.} \label{fig:sector}  \end{figure}

\subsection{The zeta regularized determinant of the Laplace operator} \label{ss:defining_determinant}
Our sign convention for the Laplace operator in dimension two in local coordinates is
\begin{equation} \label{laplacesign} \Delta_g = - \frac{1}{\sqrt{\det(g)}} \sum_{i,j=1} ^2 \pa_i g^{ij} \sqrt{\det(g)} \pa_j.
\end{equation}
Here we restrict our attention to the Friedrichs extension of the Laplace operator, noting that in two dimensions, this is equal to the Dirichlet extension \cite{domains}.  For surfaces with conical singularities and no boundary components, the eigenvalues of the Laplacian begin with $0$ and increase towards $\infty$.  For surfaces with conical singularities and at least one smooth boundary component as well as for curvilinear polygonal domains, the eigenvalues also tend to $\infty$, but in this case they are all positive.  We denote the entire collection of eigenvalues which comprise the spectrum of the Laplacian, by $\{ \lambda_k \}_{k \geq 0}$.  Then, there is an associated spectral zeta function,
\[ \zeta_{g} (s) := \sum_{\lambda_k \neq 0} \lambda_k ^{-s}. \]

One also has a corresponding heat operator and heat kernel, the Schwartz kernel of the fundamental solution to the heat equation. The trace of the heat operator, $\tr (e^{-t \Delta_g})$ is then expressed in terms of the eigenvalues,
\[ \tr (e^{-t \Delta_g}) = \sum_{k \geq 0} e^{-\lambda_k t}. \]
It is related to the spectral zeta function by
\begin{equation} \label{mellin} \zeta_{g} (s) = \frac{1}{\Gamma(s)} \int_0 ^\infty t^{s-1} \Tr \left( e^{-t \Delta_g}  - P_{\Ker (\Delta_g)} \right) dt = \frac{1}{\Gamma(s)} \int_0 ^\infty t^{s-1} \sum_{\lambda_k > 0} e^{-\lambda_k t}dt. \end{equation}
Above, $P_{\Ker (\Delta_g)}$ is the projection onto the kernel of the Laplacian.
The eigenvalues, $\{\lambda_k\}_{k \geq 0}$ grow asymptotically like $k$ as $k \to \infty$, with the precise asymptotics given by Weyl's Law \cite{weyl}.  It is then clear to see that whereas $\zeta_{g}(s)$ is well-defined for $s \in \C$ with $\Rea(s)>1$, it is less obvious that $\zeta_{g}(s)$ is well-defined for other values of $s$.

A keen observation of Ray and Singer \cite{rays} is that one may exploit the existence of the asymptotic expansion of the heat kernel for small time, together with \eqref{mellin}, to meromorphically extend the spectral zeta function.  This extension is holomorphic in a neighborhood of $s=0$, and so the determinant of the Laplacian is defined to be
\begin{equation} \label{eq:def_det0}  \det(\Delta_g) := e^{-\zeta'_{g} (0)}. \end{equation}
However, in \cite{rays}, there were no conical singularities.  The presence of even the simplest conical singularity has a profound impact on the Laplace operator.  The Laplace operator is not essentially self-adjoint, has many self adjoint extensions, and the spectrum depends on the choice of self-adjoint extension. The zeta-regularized determinant of the Laplacian also depends upon this choice \cite{gm-zeta}.

Nonetheless, it has been shown that the heat trace on a surface with conical singularities also has an asymptotic expansion for small values of $t$. In general this expansion takes the form
\begin{equation} \label{eq:heat_trace_expansion0} \tr(e^{-t\Delta_g}) = a_0 t^{-1} + a_1 t^{-\frac{1}{2}} + a_{2,0}\log(t) + a_{2,1}  + O(t^{\frac12}),  \text{ as } t\to 0, \end{equation}
see e.g. \cite[section 7, equations (7.22) and (7.23)]{BS-reso} and \cite[Theorem 5.1]{cheeger}. However, it can be shown, \cite[section 3]{King}, that for surfaces with conical singularities, the coefficient of $\log(t)$ vanishes.  
For more general curvilinear polygonal domains, this result is much more recent \cite{nrs}.  The subtlety lies in the fact that the corner need never be exactly straight.  Consequently, numerous results for `exact conical singularities' as well as results which require that the edges are straight, at least in some small neighborhood of the corner, exclude curvilinear polygonal domains.

Associated to the Riemannian metric $g$ on the surface with conical singularities, there is the scalar curvature $\Scal_g$. This  is a well defined function on $M$ but not on $\bM$ due to the presence of conical singularities.  We consider here the restriction of $\Scal_g$ to $M$ and note that this is sometimes known as the regularized scalar curvature.  This scalar curvature vanishes on the interior of an exact cone as well as the interior of a circular sector. Its connection to the constant term in the asymptotic expansion of the heat trace, $a_{2,1}$, is given explicitly in \cite[eqn(1.4)]{htap}, that we recall here:
\[a_{2,1} = \frac{1}{12 \pi} \left( \int_{M}  K_g  dA_{g} 
+ \sum_{j=1}^{\ell} \int_{\pa M} k_{g,j} ds_g \right) 
+ \sum_{i=1} ^m  \frac{ (2\pi)^2 - \gamma_i^2}{24 \pi \gamma_i}. \]  
Above, $K_g$ is the Gaussian curvature of $M$, $k_{g,j}$ is the geodesic curvature of the $j$-th boundary component, and $\ell$ is the number of boundary components. 

In recent years there has been progress towards understanding the behavior of the determinant of certain self-adjoint extensions of the Laplace operator on  surfaces with conical singularities. This progress represents different aspects studied by numerous authors; a non-exhaustive list includes the works of Kokotov \cite{kokotov1, kokotov2}, Hillairet and Kokotov \cite{hk}, Kokotov and Koronkin \cite{KoKo}, Kalvin \cite{kalvin1, kalvin2}, Kirsten et al \cite{kirstenloyapark1, kirstenloyapark2}, Loya et al \cite{LoMcDPa}, Spreafico \cite{Spre}, and Sher \cite{Sher}.

In \cite{kalvin2}, the author proves a Polyakov formula for surfaces with conical singularities with and without boundaries.  In Remark \ref{r:cwKw} we explain how that work differs from ours.  Previously, using heuristic arguments, \cite[equation (51)]{AuSa} computed a formula for the contribution of the corners to the variation of the determinant on a polygon. Here we use different techniques to rigorously prove both the differentiated and integrated Polyakov formula for surfaces with conical singularities, boundary, and curvilinear polygonal domains in surfaces.

To study the variation of the determinant in this geometric generality, refined information concerning the behavior of the heat kernel itself, not only its trace, is required.  Although it has been widely assumed that a heat trace expansion of the form \eqref{eq:heat_trace_expansion0} holds for curvilinear polygonal domains, a rigorous proof even in the planar Dirichlet case was not given until \cite{corners}. Similar results hold for Neumann boundary conditions; see \cite{htap}.  For curvilinear polygonal domains in the plane and in surfaces, the existence of an asymptotic expansion of the heat trace for small times demonstrated in \cite{nrs} allows one to extend the zeta function to a meromorphic function on the complex plane which is regular at $s=0$ and define the determinant of the Laplacian as in \eqref{eq:def_det0}.  Moreover, the microlocal construction of the heat kernel in \cite{nrs} is a key technical ingredient in our present work to study the variation of the determinant.

\subsection{Main results} \label{ss:results}
Our first results are Polyakov formulas for surfaces with isolated conical singularities, with or without smooth boundary components, and for conformal factors that are smooth up to the cone points and the boundary.  

\begin{theorem} \label{t:smoothcf-nb} Let $(\bM, g)$ be a surface that has isolated conical singularities at points $ \wp= \{p_1, \ldots, p_m\}$ with corresponding angles $\{\gamma_1, \ldots, \gamma_m\}$ and such that $\pa \bM =\emptyset$.  Let $M = \bM \setminus \wp$. Let $\Delta_g$ be the Friedrichs extension of the Laplacian with respect to the Riemannian metric, $g$.  Let $\{ h_u = e^{2 \varphi_u} g\}_{u\in (-\eps, \eps)}$ be a smooth one-parameter family of conformal metrics for a fixed $\eps > 0$, i.e., we assume that the functions $\varphi_u (z)$ depend analytically on the parameter $u$, and that both $\varphi_u (z)$, and $\pa_u \varphi_u (z)$ are smooth functions on $\bM$. In particular they are smooth up to and including all cone points. For a given metric $*$ on $\bM$, let $\Scal_*$ denote its scalar curvature on $M$, $dA_*$ denotes the corresponding area form, and $A_*$ denote the area of $(\bM, *)$.
Then, we have the variational Polyakov formula:
\beq  \left . \frac{\pa}{\pa u} \left( -  \log \det(\Delta_{h_u}) \right) \right|_{u=0} &= & \int_{M} 2 \dot \varphi_0(z) \left(\frac{\Scal_{h_0}(z)}{24 \pi} -\frac{1}{A_{h_0}}\right) dA_{h_0} \nn \\ & + & \sum_{i=1} ^m  \dot \varphi_0 (p_i) \frac{(2\pi)^2 - \gamma_i^2}{12 \pi \gamma_i}, \nn \eeq 

\noindent where $\dot \varphi_0 := \pa_u \varphi_u \vert_{u=0}$. Assume now that the conformal factors are of the form
\[ \varphi_u = \varphi_0 + u \eta, \]
for functions $\varphi_0$ and $\eta$ that are smooth up to all the conical points.  Then for the metric $h_0 = e^{2 \varphi_0} g$, we have the integrated Polyakov formula
\beq 
\log \det (\Delta_{h_0}) - \log \det (\Delta_{g})
&= & -\frac{1}{12 \pi} \int_{M} \Scal_{g}(z) \ \varphi_0(z) \ dA_{g} \nn \\ 
&-& \frac{1}{12\pi}  \int_{M} \vert \nabla_{g}\varphi_0(z) \vert^{2} \ dA_{g} \nn \\
& +&  \log(A_{h_0}) -  \log(A_g)
-  \sum_{i=1} ^m  \varphi_0 (p_i)  \frac{ (2\pi)^2 - \gamma_i ^2}{12 \pi \gamma_i} . \label{eq:IPf-swcsnb}
\eeq 
\end{theorem}

Next, we consider surfaces that may have both smooth boundary components as well as conical singularities.
\begin{theorem} \label{t:smoothcf-b}
Let $(\bM, g)$ be a surface that has isolated conical singularities at points $\wp = \{p_1, \ldots, p_m\}$ with corresponding angles $\{\gamma_1, \ldots, \gamma_m\}$ and with smooth boundary $\pa \bM \neq \emptyset$.  Let $M = \bM \setminus (\wp \cup \pa \bM)$.  Let $\Delta_g$ be the Friedrichs extension of the Laplacian with respect to the Riemannian metric, $g$.  Let $\{ h_u = e^{2 \varphi_u} g\}_{u\in (-\eps, \eps)}$ be a smooth one-parameter family of conformal metrics for a fixed $\eps > 0$, i.e., we assume that the functions $\varphi_u (z)$ depend analytically on the parameter $u$, and that both $\varphi_u (z)$, and $\pa_u \varphi_u (z)$ are smooth functions on $\bM$. In particular they are smooth up to and including all cone points and all boundary components. For the Riemannian metric $h_0$, let $\frac{\pa \psi}{\pa n_{h_0}}$ denote the normal derivative of the function $\psi$ at the boundary, $k_{h_0}$ denote the geodesic curvature of the boundary, and $dx_{h_0}$ denote the length measure on the boundary.

Then, we have the variational Polyakov formula:
\begin{multline*} \left . \frac{\pa}{\pa u} \left( -\log \det(\Delta_{h_u})\right) \right|_{u=0} =  \int_{M}  \dot \varphi_0(z) \left(\frac{\Scal_{h_0}(z)}{12  \pi}\right) dA_{h_0} \\
+ \frac{1}{6\pi} \int_{\pa M} \dot \varphi_0 (x) k_{h_0} (x) dx_{h_0} + \frac{1}{4\pi} \int_{\pa M} \frac{\pa \dot \varphi_0}{\pa n_{h_0}} (x) dx_{h_0}
+ \sum_{i=1} ^m  \dot \varphi_0 (p_i) \frac{(2\pi)^2 - \gamma_i^2}{12 \pi \gamma_i}.  \end{multline*}

Assume now that the conformal factors are of the form
\[ \varphi_u = \varphi_0 + u \eta, \]
for a function $\eta\in \cC^{\infty}(\bM)$. Then for the metric $h_0 = e^{2 \varphi_0} g$, we have the integrated Polyakov formula
\beq \label{eq:IPf-swcsb}
\log \det (\Delta_{h_0}) - \log \det(\Delta_g)
 &=&  -  \frac{1}{12\pi} \int_{M} \vert \nabla_{g} \varphi_0 (z)  \vert^{2} \ dA_{g} \nn \\ 
 &-& \frac{1}{12 \pi}  \int_{M} \varphi_0(z) \Scal_{g} (z) dA_{g} \nn \\
 &-&  \frac{1}{4\pi} \int_{\pa \bM} \frac{\pa \varphi_0(x)}{\pa n_g}dx_g \nn \\ 
 &-& \frac{1}{6\pi} \int_{\pa \bM}  \varphi_0(x) k_{g}(x) dx_g - \sum_{i=1} ^m  \varphi_0 (p_i) \frac{(2\pi)^2 - \gamma_i^2}{12 \pi \gamma_i}.
 \eeq 
\end{theorem}

\begin{remark} \label{r:cwKw}  We note that the integrated Polyakov formulae in Theorems \ref{t:smoothcf-nb} and \ref{t:smoothcf-b} can be obtained as consequences of \cite[Corollary 1.3]{kalvin2} and \cite[Corollary 1.3.2]{kalvin2}, respectively.  That work is more general in the sense that the angles are allowed to change.  Nonetheless, as V. Kalvin observes in his paper \cite[p.32]{kalvin2}, our methods are different, so our results may nonetheless be of independent and complementary interest.  Kalvin's proof is based in the BFK formulae for the determinant on a surface with conical singularities with conformal factors that have a logarithmic singularity at the cone points.  He also uses an asymptotic formula for the determinant of the Laplacian of a metric with a cone singularity on the disk of radius $\varepsilon$ as $\varepsilon \to 0^+$. For this last part, he expresses the spectral zeta functions in terms of the corresponding resolvent operator. In contrast, our formulae presented here are  {\bf variational} Polyakov's formulae. They come from a variational principle, that keeps the geometrical meaning of the problem all along the process, that is why we consider it to be relevant. Moreover, it may be possible to build upon our techniques to extend the class of metrics allowed in \cite{kalvin2}. Finally, we note that in Appendix \ref{ss:vpcone}, we prove that equations (\ref{magic1-npij}) and (\ref{miracle}) below can also be obtained starting from a variational principle.
\end{remark}

We next obtain the variational Polyakov formula and the integrated Polyakov formula for curvilinear polygonal domains in surfaces.  

\begin{theorem} \label{thm:domains} Let $\Omega$ be a curvilinear polygonal domain on the Riemannian manifold $(M,g)$, as in Definition \ref{def:curvpoly}. $\Omega$ has piecewise smooth boundary and finitely many corners at points $\{ p_1, \ldots, p_m\}$ with corresponding angles $\{\gamma_1, \ldots, \gamma_m\}$.  Let $\Delta_g$ be the Dirichlet-Friedrich extension of the Laplacian in $\Omega$.  Let $\{ h_u = e^{2 \varphi_u} g\}_{u \in (-\eps, \eps)}$ be a smooth one-parameter family of metrics conformal to $g$ for a fixed $\eps > 0$.  We assume that $\varphi_u (z)$ depends smoothly on the parameter, $u$, and that both $\varphi_u (z)$, and $\pa_u \varphi_u (z)$ are smooth functions on $M \supseteq \overline \Omega$, in particular smooth up to and including all corner points and boundary components of $\Omega$.  Then, we have the variational Polyakov formula:
\begin{multline*} \left . \frac{\pa}{\pa u} \left( -  \log \det(\Delta_{h_u}) \right)\right|_{u=0} =  \frac{1}{12 \pi} \int_{\Omega} \dot \varphi_0(z) \left(\Scal_{h_0}(z) \right) dA_{h_0} \\ +\frac{1}{6\pi} \int_{\pa \Omega} \dot \varphi_0 (x) k_{h_0} (x) dx_0 + \frac{1}{4\pi} \int_{\pa \Omega} \frac{\pa \dot \varphi_0}{\pa n_{h_0}} (x) dx_0+ \sum_{i=1} ^m  \dot \varphi_0 (p_i) \frac{\pi^2 - \gamma_i^2}{12 \pi \gamma_i}.  \end{multline*}
Assume now that the conformal factors are of the form
\[ \varphi_u = \varphi_0 + u \eta, \]
for a function $\eta$ which is smooth on $M \supseteq \overline \Omega$.  Then for the metric $h = e^{2 \varphi_0} g$, we have the integrated Polyakov formula
\begin{multline*}
\log \det (\Delta_{h_0}) - \log \det(\Delta_g)   =  -  \frac{1}{12\pi} \int_{\Omega} \vert \nabla_{g} \varphi_0(z)   \vert^{2} \ dA_{g} - \frac{1}{12 \pi}  \int_{\Omega} \varphi_0(z) \Scal_{g}(z) dA_{g}  \\
- \frac{1}{6\pi} \int_{\pa \Omega} \varphi_0(x) k_{g}(x) dx_g  - \frac{1}{4\pi} \int_{\pa \Omega} \frac{\pa \varphi_0(x)}{\pa n_g}dx_g - \sum_{i=1} ^m  \varphi_0 (p_i) \frac{\pi^2 - \gamma_i^2}{12 \pi \gamma_i}.
\end{multline*}
\end{theorem}

\begin{remark}
In Theorem \ref{thm:domains}, we prove the L\"uscher-Symanzik-Weiss-Polyakov relation from \cite{lusc80-173-365}.  We note that this agrees with the result obtained by Dowker [eq. (8) in \cite{dowk94-11-557}] in a formal computation.
\end{remark}

Here, we obtain an explicit formula for the determinant for both finite cones and finite sectors. 

\begin{theorem} \label{t:sectors}
Let $S_\alpha$ be a circular sector of opening angle $\alpha$ and radius one.  Assume the Dirichlet boundary condition for the Laplacian.  Then, we have
\begin{multline} \label{eq:magicvar1}
- \log(\det(\Delta_{S_\alpha})) = \frac 1 4 (\gamma_e +2) +\frac 5 {24 \pi} \alpha + \frac 1 {12} \left( \gamma_e - \log 2\right) \left( \frac \pi \alpha + \frac \alpha \pi \right) \\
 +\int\limits_1^\infty  \frac 1 t \,\, \frac 1 {e^{\frac \pi \alpha t} -1} \,\, \frac 1 {e^t-1} dt + \int\limits_0^1  \frac 1 t \,\,\left( \frac 1 {e^{\frac \pi \alpha t} -1} \,\, \frac 1 {e^t-1} - \frac{ \alpha }{\pi t^2} + \frac{\pi + \alpha }{2\pi t} - \frac{\pi^2 + 3\pi \alpha + \alpha^2}{12\pi \alpha} \right) dt .\end{multline}
Above, $\gamma_e$ is the Euler-Mascheroni constant.  The derivative with respect to the angle,
\beq 
 \frac{d}{d\alpha} (- \log \det (\Delta_{S_\alpha})) & = &  \frac{5}{24 \pi} + \frac{1}{12} (\gamma_e - \log 2) \left( - \frac{\pi}{\alpha^2} + \frac 1 \pi \right) + \int_1 ^\infty  \frac{ \frac{\pi}{\alpha^2} e^{\frac \pi \alpha t} }{(e^{\pi t/\alpha} - 1)^2} \frac{dt}{e^t - 1} \nn \\
& +& \int_0 ^1 \left( \frac{ \frac{\pi}{\alpha^2} e^{\frac \pi \alpha t} }{(e^{\pi t/\alpha} - 1)^2} \frac{1}{e^t - 1}  -  \frac{1}{t} \left( \frac{1}{\pi t^2} - \frac{1}{2\pi t} + \frac{1}{12 \pi}  - \frac{\pi}{12\alpha^2} \right) \right)dt. \label{eq:magic_angle} \eeq
When $\alpha$ is not equal to $\frac \pi j$ for any integer $j$, then
\beq 
 \frac{d}{d\alpha} (- \log \det (\Delta_{S_\alpha})) &= &\frac{1}{3\pi} + \frac{\pi}{12 \alpha^2}  -  \sum_{k=1} ^{\ceil*{ \frac{\pi}{2\alpha} -1}}  \frac{\gamma_e + \log |\sin(k \alpha)|}{2\pi \sin^2 (k \alpha)}  \nn \\
 & +& \frac{1}{\alpha} \sin\left( \frac{\pi^2}{\alpha} \right) \int_\R \frac{ - \log 2 + 2 \gamma_e + \log(1+\cosh(s))}{8 \pi (1+\cosh(s)) (\cosh(\pi s/\alpha) - \cos(\pi^2/\alpha))} ds. \nn \eeq
If  $\alpha = \frac \pi j$ for some integer $j > 1$,  then
\beq \label{eq:dalogdetLaplSpij}
 \frac{d}{d\alpha} (- \log \det (\Delta_{S_\alpha})) &=& \frac{1}{3\pi} + \frac{\pi}{12\alpha^2} - \frac{\gamma_e}{12 \pi} \left( \frac{\pi^2}{\alpha^2} - 1 \right)  \\
 &-& \sum_{k=1} ^{\lceil \frac{\pi}{2\alpha} - 1 \rceil}  \frac{ \log \left|\sin( k \alpha) \right|}{2 \pi \sin^2 \left(k \alpha \right)}. \nn \eeq
\end{theorem}

Notice that equation (\ref{eq:dalogdetLaplSpij}) implies that at angles of the form $\alpha = \frac \pi j$, $\frac{d}{d\alpha} (- \log \det (\Delta_{S_\alpha})) > 0$. Hence  $\frac{d}{d\alpha} (\log \det (\Delta_{S_\alpha})) < 0$ and $\frac{d}{d\alpha} (\det (\Delta_{S_\alpha})) < 0$. Moreover,
$$\lim_{j\to \infty} \frac{d}{d\alpha} (\log \det (\Delta_{S_{\alpha}}))\vert_{\alpha=\pi/j} = -\infty.$$ Therefore, by continuity, we have that $\lim_{\alpha\to 0} \frac{d}{d\alpha} (\log \det (\Delta_{S_{\alpha}})) = -\infty.$ The same holds for cones, due to equation (\ref{miracle}) below.  

\begin{remark}
For several values of $j$, using the approach of Dowker as in \cite{dowk94-11-557}, the results above have been confirmed  \cite{privatecommun}.  This result notably allows for the variation of conical points that lie \em on the boundary \em of the domain, which is not allowed in \cite{kalvin2}.  Moreover, we note that our variational formula for angles of the form $\frac \pi j$ is in a \em fully explicit closed form;  \em there are no lingering infinite sums, integrals, or values of special functions requiring computation.
\end{remark}

We next obtain an explicit formula for the determinant of finite cones as well as the variational Polyakov formula for the determinant under variation of the cone angle.

\begin{theorem} \label{cthm:detcone}
Let $C_{\alpha}$ be a cone with angle $\alpha$, and height equal to one.  Then we have
\begin{align*}
- \log(\det(\Delta_{C_{\alpha}})) =& - \frac 1 2 \log(2\pi) + \frac 1 2 (\gamma_e +2) +\frac 5 {24 \pi} \alpha + \frac 1 {6} \left( \gamma_e - \log 2\right) \left( \frac{2 \pi}{ \alpha} + \frac \alpha {2 \pi} \right)  \notag \\
 & +2\int\limits_1^\infty  \frac 1 t \,\, \frac 1 {e^{\frac {2\pi} \alpha t} -1} \,\, \frac 1 {e^t-1}  dt \\
 & +2 \int\limits_0^1  \frac 1 t \,\,\left( \frac 1 {e^{\frac {2 \pi} \alpha t} -1} \,\, \frac 1 {e^t-1} - \frac{ \alpha }{2 \pi t^2} + \frac{\pi + \alpha/2 }{2\pi t} - \frac{\pi^2 + 3\pi \alpha/2 + \frac{\alpha^2}{4}}{6 \pi \alpha} \right) dt. \notag
\end{align*}
The derivative with respect to the angle for a cone of angle $\alpha$ is
\begin{align*} \frac{d}{d\alpha} \left(- \log(\det(\Delta_{C_{\alpha}})) \right) =& \frac{5}{24 \pi} + \frac{1}{6} (\gamma_e - \log 2) \left( - \frac{2 \pi}{\alpha^2} + \frac 1 {2\pi} \right) + 2 \int_1 ^\infty   \frac{ \frac{2 \pi}{\alpha^2} e^{\frac {2\pi} \alpha t} }{(e^{2 \pi t/\alpha} - 1)^2} \frac{1}{e^t - 1} dt \notag \\
& + 2 \int_0 ^1  \left( \frac{ \frac{2 \pi}{\alpha^2} e^{\frac {2\pi} \alpha t} }{(e^{2 \pi t/\alpha} - 1)^2} \frac{1}{e^t - 1}  - \frac 1 t \left( \frac{1}{2 \pi t^2 } - \frac{1}{4\pi t} - \frac{\pi}{6\alpha^2} + \frac{1}{24 \pi} \right) \right)dt.  \end{align*}
When the cone angle $\alpha$ is not equal to $\frac {2\pi}{ j} $ for any integer $j$, the equation above becomes
\beq \frac{d}{d\alpha} \left(- \log(\det(\Delta_{C_{\alpha}})) \right) &= & \frac{1}{3\pi} + \frac{\pi}{3 \alpha^2}  -  \sum_{k=1} ^{\ceil*{ \frac{\pi}{\alpha} -1}}  \frac{\gamma_e + \log |\sin(k \alpha/2)|}{2\pi \sin^2 (k \alpha/2)} \notag \\
& + & \frac{1}{\alpha} \sin\left( \frac{2 \pi^2}{\alpha} \right) \int_\R \frac{ - \log 2 + 2 \gamma_e + \log(1+\cosh(s))}{4 \pi (1+\cosh(s)) (\cosh \frac{2\pi s}{\alpha} - \cos \frac{2\pi^2}{\alpha})} ds. \label{magic1-npij} \eeq

If the angle $\alpha = \frac{2\pi}{j}$ for some integer $j$, then we have 
\beq 
\frac{d}{d\alpha} \left(- \log(\det(\Delta_{C_{\alpha}})) \right) &=&  \frac{1}{3\pi} + \frac{\pi}{3\alpha^2} - \frac{\gamma_e}{12 \pi} \left( \frac{4\pi^2}{\alpha^2} - 1 \right)  \nn \\
&-& \frac{1}{2\pi} \sum_{k=1} ^{\lceil \frac{\pi}{\alpha} - 1 \rceil}  \frac{ \log \left|\sin\left( k \alpha/2 \right)\right|}{\sin^2 \left(k \alpha/2 \right)}. \label{miracle} \eeq 
\end{theorem}

\begin{remark} A formula for the zeta function and its derivative on a finite cone has been computed in \cite[Theorem 1]{Spre} by Spreafico using independent methods.  In the proof of Theorem \ref{cthm:detcone}, we show that this formula coincides with ours and is obtained virtually effortlessly from \cite{bord96}. 
\end{remark}

\subsection{Organization}
Theorems \ref{t:smoothcf-nb}, \ref{t:smoothcf-b}, and \ref{thm:domains} are proven in section \ref{s:smoothcf}. In \S \ref{s:explicit}, we obtain the explicit expression for the determinant on finite circular sectors and finite cones.  We then combine these expressions with results from the literature for the Barnes zeta function to obtain the first and second displayed equations in Theorems \ref{t:sectors} and \ref{cthm:detcone}.  The proofs of these theorems are completed by recalling the variational formula of \cite{AldRow} and manipulating that formula to obtain the simplified versions presented here.  In Appendix \ref{s:heatcalcapp}  we compute the contribution to the short time asymptotic expansion for the heat trace due to the conical singularities for surfaces with conical singularities.  The calculation agrees with that of \cite{cheeger}, but is obtained by a completely independent method.  In Appendix \ref{ss:vpcone}, we provide an alternative method for computing the angular variation of the determinant on a finite cone by considering a conformal change of metric in the style of \cite{AldRow}; this coincides with the expression obtained by the  independent method used to prove Theorem \ref{cthm:detcone}

\section*{Acknowledgements}
C.L. Aldana was partially supported by the Fonds National de la Recherche, Luxembourg 7926179.  JR \& CLA gratefully acknowledge  the National Science Foundation award DMS-1440140 which supported our time at the Mathematical Sciences Research Institute in Berkeley, California during the Fall 2019 semester.  JR is supported by Swedish Research Council Grant 2018-03873.  We are grateful to Stuart Dowker, Yilin Wang, and Eveliina Peltola for insightful discussions and correspondence, as well as the referee for constructive critiques that improved the quality of the paper.

\section{Polyakov formula for surfaces with conical singularities, boundary, and curvilinear polygonal domains}  \label{s:smoothcf}
Here we prove Theorems \ref{t:smoothcf-nb} and \ref{t:smoothcf-b}. That is, we prove a Polyakov formula for surfaces with isolated conical singularities and smooth boundary components, as given in Definition \ref{def:conicmet}.  We then obtain Theorem \ref{thm:domains} as an immediate consequence. Although we separated the cases with boundary and without boundary, the arguments are identical until the end at which point we distinguish between the two.  The proof of the Polyakov formulas presented in this section follows the same lines as the corresponding proofs in \cite{OPS}, \cite{Aldana-se} and \cite{AldRow}, but adapted to our case; for the details we refer to these references. Recall that on these surfaces, the set of conical singularities and the set of smooth boundary components must have empty intersection. The conformal factors considered here are assumed to be smooth all the way up to the conical points and up to the boundary.

Let $(\bM, g)$ be a surface with conical singularities and smooth boundary. Denote by $\Delta_g$ the Friedrichs extension of the Laplace operator associated to $g$. If there are boundary components, assume the Dirichlet boundary condition at all of them. Since in dimension two, the Friedrichs extension coincides with the Dirichlet extension \cite{domains}, we just refer to this extension as the Dirichlet extension or the Friedrichs extension.

If there are only conical points and $\bM$ has no boundary, then $\Delta_g$ has a non-trivial kernel, $\Ker(\Delta_g)$, which consists of constant functions.  If $\bM$ has at least one boundary component, then the kernel consists only of  the constant function zero and has dimension zero. In all cases, let $H_g(t, z, z')$ denote the heat kernel associated to $\Delta_{g}$. The heat kernel is the Schwartz kernel of the heat operator, $e^{-t \Delta_g}$.  The trace of the heat operator can be expressed as
\[ \tr(e^{-t{\Delta_g}}) = \int_\bM H_g(t, z, z) dA_g(z). \]
This expression can be used to express the spectral zeta function as
\begin{align*} \zeta_g (s) &= \frac{1}{\Gamma(s)} \int_0 ^\infty t^{s-1} \tr(e^{-t{\Delta_g}}-P_{\Ker(\Delta_g)}) dt \\
&= \frac{1}{\Gamma(s)} \int_0 ^\infty t^{s-1} \int_\bM \left( H_g(t, z, z) - \dim\Ker(\Delta_g) \right)  dA_g(z) \ dt,
\end{align*}
where $P_{\Ker(\Delta_g)}$ denotes the projection on the kernel of $\Delta_g$.

Let $\eps>0$, and let $\{ h_u = e^{2\varphi_u} g\}_{u\in (- \eps, \eps)}$ be a one-parameter family of metrics that are conformal to $g$. In this section, we assume that, for each $u$, $\varphi_u \in \cC^{\infty}(\bM)$; that is, $\varphi_u$ is a smooth function on $M$ that is smooth up to the cone points and the smooth boundary components. Moreover, we assume that the dependence on the parameter $u$ is also smooth. Under these conditions, $(\bM,h_{u})$ is a surface with conical singularities at the same points as $(\bM,g)$, and with the same conical angles at the corresponding points, since smooth conformal factors do not change angles as observed in \cite{AldRow}. 

The Dirichlet extensions of the corresponding Laplace operators satisfy
\[ \Delta_{h_u} = e^{-2\varphi_u} \Delta_g. \]
The area elements transform as
\[ dA_{h_u} = e^{2\varphi_u} dA_g,\]
and for the scalar curvatures we have 
\begin{equation} \label{eq:scalarcurvatures} \Scal_{h_u} = e^{-2\varphi_u} \big(\Scal_{g} + 2 \Delta_g \varphi_u \big). \end{equation}

To prove Polyakov's formula we will require heat kernel estimates.
\subsection{Heat kernel estimates} \label{hkest}
The heat kernel estimates we require follow easily  from \cite{davies}, \cite{acm}, and \cite{mooers}.

\begin{prop} \label{pr-heat2}
Let $(\bM,g)$ be a surface with conical singularities and smooth boundary as given in Definition \ref{def:conicmet}. Let $\{ \varphi_u (z) \}_{u\in (-\eps, \eps)}$ be a family of functions that are smooth on $(-\eps, \eps) \times \bM$ for some $\epsilon > 0$.  Let $T>0$, then there is a constant $C=C(T)>0$ such that the heat kernel, $H_u$, associated to the Dirichlet extension of the Laplacian on $(M, h_u = e^{2\varphi_u} g)$ satisfies the estimates
\begin{eqnarray*}
\left| H_u (t,z,z') \right| &\leq& \frac{C}{t},\\
\left| \pa_t H_u (t,z,z') \right| &\leq& \frac{C}{t^2},
\end{eqnarray*}
for all $z,z'\in \bM$, and all $t \in (0, T)$. These estimates are uniform for all $u \in [-\epsilon/2, \epsilon/2]$.
Moreover, these estimates also hold for the Dirichlet heat kernel on curvilinear polygonal domains in surfaces.
\end{prop}

\begin{proof}
Each conical singularity on $\bM$ has a neighborhood,
\[ \cN_i  \cong [0, \eps_i]_r \times \bS^1, \]
on which the metric $h_u = e^{2\varphi_u} g$ is of the form
\[ \left . h_u \right|_{\cN_i}  =  \left. e^{2\varphi_u} g \right|_{\cN_i} = e^{2\varphi_u} (dr^2 + r^2 \omega_i(r)). \]
Consequently, $(\bM, e^{2\varphi_u} g)$ is of the same type, in the sense that it has isolated conical singularities and has smooth boundary components whenever $M$ does. Moreover, if $(\bM, g)$ is a curvilinear polygonal domain contained in a smooth surface, then the same considerations imply that $(\bM, e^{2 \varphi_u} g)$ is also.  When the conical singularities are of the exact type, $(\bM, e^{2\varphi_u} g)$ is an
example of the much more general stratified spaces considered in \cite{acm}.  The Dirichlet heat kernels, $H_u$, therefore satisfy the estimate (2.1) on p. 1062 of \cite{acm}.  This estimate is
\begin{equation} \label{hke1} H_u(t, z, z') \leq C t^{-1}, \quad \forall z, z'\in \bM, \quad \forall t \in (0, 1), \end{equation}
since the dimension $n=2$.   When the conical singularities are not exact,  the calculations of Mooers, \cite[p. 13]{mooers} show that we obtain the same estimate.  There, she obtained estimates comparing heat kernels with general conical singularities to model heat kernels which have exact conical singularities.

Next, we apply the results by E.B. Davies in \cite{davies} which hold for the Laplacian on a general Riemannian manifold whose balls are compact if the radius is sufficiently small. These minimal hypotheses are satisfied for surfaces with generalized isolated conical singularities.  This estimate was demonstrated for sectors in \cite[Proposition 4]{AldRow}, and the exact same argument may be copy-pasted here so we simply summarize the result:
\[ \left| \pa_t H_u(t,z,z') \right| \leq C t^{-2}, \quad \forall t \in (0,T), \quad \forall z, z' \in  \bM. \]
Since the conformal factors $\varphi_u(z)$ are smooth for $u \in (-\epsilon, \epsilon)$ and for $z \in \bM$, by the compactness of $[-\epsilon/2, \epsilon/2]$ we obtain the uniformity of all estimates obtained here for $u \in [-\epsilon/2, \epsilon/2]$.
\end{proof}

\begin{remark}
It is well known that the estimate for the time derivative of the heat kernel implies the following estimate for the Laplacian of the heat kernel
\[ \left| \Delta_u H_u(t,z,z') \right| \leq  C t^{-2}, \]
for any $0<t<T$, and $z, z' \in  \bM$, for a constant $C>0$ depending on $T$.
\end{remark}

In addition to the estimates on the heat kernel mentioned above, we need the existence of an asymptotic expansion of the trace of the heat operator for small times. When $\bM$ has no boundary, $\pa \bM =\emptyset$, but $\bM$ still has isolated conical singularities, this follows from the results in \cite{mooers}; see also \cite{les}, \cite{seeley}, \cite{Sher}. When $\pa \bM \neq \emptyset$, and $\bM$ does not have any conical singularities, the existence of the asymptotic expansion of the heat trace for small times is well known, \cite{Kac}, \cite{alv}, \cite{Gilkey}.
Since the asymptotic expansion of the heat kernel can be found from those of the local models by using a parametrix construction, it is not problematic to obtain the expansion from the corresponding theorems in the references listed above.  However, since we allow in addition to conical singularities and smooth boundary components, curvilinear polygons, we refer for the existence of such an expansion to Theorem 5.5. in \cite{nrs} that covers the most general case.

\begin{cor} \label{cor:exp1} 
Let $(\bM,g)$ be a surface with finitely many conical singularities and smooth boundary $\pa \bM$. The heat operator associated to a metric $h = e^{2\varphi}  g$ conformal to $g$ with $\varphi\in \cC^{\infty}(\overline{\bM})$
is trace class for all $t>0$, and the trace has an expansion of the form
\[ \Tr_{L^2 (\bM, h)} \Big(e^{-t \Delta_{h} }  -P_{\Ker(\Delta_{h})} \Big) \sim
a_0t^{-1} + a_1 t^{-\frac 1 2} + a_{2,1} + O(t^{\frac 1 2} \log t), \quad t \downarrow 0. \]
Let $\psi\in \cC^{\infty} (\overline{\bM})$, and let $\cM_{\psi}$ denote the operator multiplication by $\psi$, then as $t \downarrow 0$, the trace of the operator, $\cM_{\psi}\Big(e^{-t \Delta_{h} }  -P_{\Ker(\Delta_{h})} \Big),$
has the following expansion
\[ \Tr_{L^2 (\bM, h)} \Big(\psi \big(e^{-t \Delta_h }  -P_{\Ker(\Delta_h)}\big) \Big) \sim   a_0(\psi)  t^{-1} + a_1(\psi) t^{-\frac 1 2} + a_{2,1}(\psi) \] 
\[ + O(t^{\frac 1 2} \log(t) ), \quad t \downarrow 0. \] 
\end{cor}

\begin{proof}
The existence of the short time expansions can be shown in a few different ways.  For surfaces with isolated exact flat conical singularities, the first result of which we are aware is due to Fursaev, \cite[Eqn 2.2]{Fursaev}.  Later but in a more general setting, we see that the result follows from Mooers \cite[Thm 3.1]{mooers} in the case of isolated conical singularities but no smooth boundary components.  Allowing both smooth boundary components and isolated conical singularities, the first statement of this  corollary is an immediate consequence of \cite[Theorem 5.5]{nrs}.

The second statement is also an almost immediate consequence of Theorem 4.1 in  \cite{nrs}. The heat kernel, $H_h$ for the heat operator $e^{- t \Delta_h}$ is shown to be a polyhomogeneous conormal distribution on the heat space, which is a manifold with corners created by blowing up $\bM \times \bM \times [0,1)$.  For the details, we refer to \cite{nrs}; see also \cite[\S 3]{AldRow}.  Here we provide just those details needed to prove the corollary.  For a boundary face defined by the function $x$, we say that a function $w$ is polyhomogeneous up to the boundary face if it admits an expansion near $x=0$ of the form
\[ w \sim \sum_{ \Rea s_j \to \infty} \sum_{p=0} ^{p_j} x^{s_j} (\log x)^p a_{j,p} (x,y), \quad a_{j,p} \in \cC^\infty. \]
Above, $\{s_j\}_{j \in \N} \subset \C$, whereas the second sum is over a finite set of non-negative integers for each $j$.  The heat space has various boundary faces at $t=0$.  Any function which is smooth on $\bM$ and smooth up to the cone points as well as the smooth boundary components lifts to the heat space to also be a smooth function which is independent of time.  The heat kernel $H_h$ is polyhomogeneous, and the powers of $t$ in its polyhomogeneous expansion at the boundary faces at $t=0$ were computed in \cite[Theorem 4.1]{nrs}.  Consequently, the product of the lift of $\psi$ as in the statement of this corollary with the heat kernel $H_h$ is polyhomogeneous conormal on the heat space.  Moreover, the expansions at the $t=0$ boundary faces simply absorb $\psi$ into the coefficient functions, $a_{j,p}$, because $\psi$ is independent of $t$.  In particular, the exponents and the powers of log are unchanged by $\psi$.  Hence, when one takes the trace by integrating along the diagonal, the powers of $t$ in the expansion are completely determined by $H_h$, independent of $\psi$.  Consequently, by \cite{nrs} (see also \cite{cheeger}, \cite{Sul}),
the trace
\[ \Tr_{L^2 (\bM, h)} \Big(\psi  \big(e^{-t \Delta_{h} }  -P_{\Ker(\Delta_{h})}\big) \Big) \sim
a_0(\psi) t^{-1} + a_1(\psi) t^{-\frac 1 2} + a_{2,1}(\psi) \] 
\[ + O(t^{\frac 1 2} \log(t)), \quad t \downarrow 0. \]
\end{proof}

Note that this coincides with the results in \cite{Sul} where metrics of the form $dr^2 + f(r)^2d\theta^2 $ are considered.

\subsection{Differentiating the family of zeta functions}  \label{s:formula1}
Let $(\bM, g)$ be a surface with conical singularities. For the moment we do not specify whether or not $\pa \bM$ is empty.
We consider the family of spectral zeta functions $\{\zeta_{h_u}\}_u$ associated to the family of metrics $\{h_u = e^{2\varphi_u} g\}_u$ on $\bM$. Here, $\{\varphi_u\}_{u}$ is a family of smooth functions on $\bM$ that are smooth up to the cone points and the smooth boundary components. Moreover, we assume that the dependence on the parameter $u$ is analytical. We are particularly interested in the case when $\varphi_u = \varphi_0 + u \psi$, for functions $\varphi_0$ and $\psi$ that are smooth on $\bM$ up to the cone points and the smooth boundary components.  Then we recall that the spectral zeta function satisfies 
\beq \zeta_{{h_u}}(s) = \frac{1}{\Gamma(s)} \int_0 ^\infty  t^{s-1} \tr \big(e^{-t\Delta_{h_u}}-P_{\Ker(\Delta_{h_u})}\big) \, dt. \label{eq:zeta_Gamma_heat} \eeq 

In order to compute the variation of the determinant, we should differentiate the family $\det(\Delta_{h_u})$ with respect to $u$ and evaluate at $u=0$. Recall that the determinant is the derivative of the meromorphic extension of the spectral zeta function at $s=0$. 
Therefore, we have a function of two variables, $\zeta(u,s) = \zeta_{{h_u}}(s)$. To obtain the local aspect of the variation of the determinant, the connection with the heat invariants, and Polyakov's formula with this method, we should perform an exchange in the order of differentiation. Hence, we need to use  
\beq \left. \frac{\pa \ } {\pa u} \left. \frac{\pa \  }{\pa s} \zeta(u,s) \right| _{s=0}  \right| _{u=0}  = \left. \frac{\pa \  }{\pa s} \left. \frac{\pa \  }{\pa u} \zeta (u,s) \right|_{u=0} \right|_{s=0}. \label{eq:swap_limits} \eeq 
As a matter of fact, using the asymptotic expansion of Corollary \ref{cor:exp1}, we prove that for each $u$ fixed, $\zeta(u,s)$ has a meromorphic continuation that is regular at $s=0$.  
Observe that the spectral zeta functions given in \eqref{eq:zeta_Gamma_heat} converge on the half plane $\Rea(s)>1$. Moreover, $\Gamma^{-1}(s)$ is an entire function with simple zeros at the non-positive integers.  The integrand in \eqref{eq:zeta_Gamma_heat} converges absolutely and uniformly on compact subsets of $\Rea(s) >1$.  The meromorphic extension of the zeta function is obtained by splitting the integral into two parts:  $\int_0 ^\delta + \int_\delta ^\infty$ for any choice of $\delta >0$.  The term $\int_\delta ^\infty$ is an entire function of $s$.  To meromorphically extend the term $\int_0 ^\delta$, the expansion obtained in Corollary \ref{cor:exp1} is inserted into the integral $\int_0 ^\delta$.  Multiplication with the factor $\Gamma^{-1}(s)$ then results in a quantity that is analytic about $s=0$.  Therefore $\zeta(u,s)$ is analytic as function of $s$ in a neighbourhood of $s=0$, uniformly and {\bf analytically} in $u$ in a compact neighbourhood of $u=0$.  We note that this meromorphic extension is independent of the choice of the parameter $\delta > 0$.  Moreover, the estimates on the heat kernels given in Proposition \ref{pr-heat2} are uniform in $u$, because all  the conformal factors considered are smooth up to the boundary and up to the conical singularities, therefore $\zeta(u,s)$ in analytic in $u$, and we can extend $\zeta(u,s)$ analytically as a function of $u\in \C$ in a small neighbourhood of $u=0$. Thus, 
$\left. \frac{\pa \  }{\pa s} \zeta(u,s) \right| _{s=0}$ is also analytic in $u$. Therefore, by classical complex analysis arguments, for example those in \cite{Huybrechts, DeBruijn}, we can exchange the order of the derivatives.  We note that this is the same argument used to define the determinant of the Laplacian through the zeta regularization process that evoques the properties of the heat trace and its asymptotic expansions. This is also the same idea used to prove variational Polyakov formulas in other cases, including closed surfaces and surfaces with hyperbolic ends as in several references including but not limited to \cite{osgo88-80-212}, \cite{OPS}, \cite{AAR}, and \cite{Aldana-se}.

Now, having justified the exchange of limits in \eqref{eq:swap_limits}, we start by differentiating the family of zeta functions with respect to the parameter $u$:
\[  \frac{\pa \ }{\pa u} \zeta_{h_u}(s) =  \frac{1}{\Gamma(s)} \int_0 ^\infty  t^{s-1} \frac{\pa }{\pa u} \tr_{L^2(\bM,h_u)} \big(e^{-t\Delta_{h_u}}-P_{\Ker(\Delta_{h_u})} \big) \, dt. \]
Here, $L^2 (\bM, h_u)$ indicates that the trace is taken on $\bM$ with respect to the area element $dA_{h_u}$.

Let us notice that for each $u$, the domain of the Dirichlet Laplacian $\Delta_{h_u}$ is equivalent to the domain of $\Delta_{h_0}$ by a bounded isometry with bounded inverse. This follows from
the fact that the conformal factors $\varphi_u \in \cC^{\infty}(\bM)$.
Since, in addition, the kernel, $\Ker(\Delta_{h_u})$, is the same for all $u$ we differentiate the heat trace as follows:
\begin{multline*}
\left. \frac{\pa}{\pa u} \tr_{L^2(\bM,h_u)} \big(e^{-t\Delta_{h_u}}-P_{\Ker(\Delta_{h_u})} \big)\right|_{u=0} =
\left. \frac{\pa}{\pa u} \tr_{L^2(\bM,h_u)} \big(e^{-t\Delta_{h_u}} \big)\right|_{u=0} \\
= - t \  \Tr_{L^2(\bM,h_0)} \Big( \big(\partial_u  \Delta_{h_u}|_{u=0}\big) e^{-t\Delta_{h_0}} \Big)
= 2t \ \Tr_{L^2 (\bM, h_0)} \left( (\pa_u \varphi_u |_{u=0})  \Delta_{h_0} e^{-t \Delta_{h_0}} \right).
\end{multline*}
On the other hand we have
\begin{multline*}
\frac{\partial}{\partial t} \Tr_{L^2 (\bM, h_0)} \Big(
(\pa_u \varphi_u |_{u=0}) \big(e^{-t \Delta_{h_0} }  -P_{\Ker(\Delta_{h_0})}\big) \Big) \\= \frac{ \partial}{\partial t} \Tr_{L^2 (\bM, h_0)} ( (\pa_u \varphi_u |_{u=0}) e^{-t \Delta_{h_0}} ) \\
= - \Tr_{L^2 (\bM, h_0)} ( (\pa_u \varphi_u |_{u=0}) \Delta_{h_0} e^{-t \Delta_{h_0}} ).
\end{multline*}
Therefore
\begin{multline}
\left. \frac{\pa}{\pa u} \tr_{L^2(\bM,h_u)} \big(e^{-t\Delta_{h_u}}-P_{\Ker(\Delta_{h_u})} \big)\right|_{u=0}\\
 = -2t\ \frac{\partial}{\partial t} \Tr_{L^2 (\bM, h_0)} \Big(
(\pa_u \varphi_u |_{u=0}) \big(e^{-t \Delta_{h_0} }  -P_{\Ker(\Delta_{h_0})}\big) \Big).
\end{multline}

Putting these equations together and integrating by parts, one obtains
\beq 
\left. \frac{\pa \ }{\pa u} \zeta_{h_u}(s) \right|_{u=0} &=& -\frac{1}{\Gamma(s)} \int_0 ^\infty  t^{s}\ \frac{\partial}{\partial t} \Tr_{L^2 (\bM, h_0)} \Big(
2 (\pa_u \varphi_u |_{u=0}) \big(e^{-t \Delta_{h_0} }  -P_{\Ker(\Delta_{h_0})}\big) \Big)    \, dt  \nn \\ 
&=& \frac{s}{\Gamma(s)} \int_0 ^\infty  t^{s-1} \Tr_{L^2 (\bM, h_0)} \Big( 2
(\pa_u \varphi_u |_{u=0}) \big(e^{-t \Delta_{h_0} }  -P_{\Ker(\Delta_{h_0})}\big) \Big)    \, dt . \label{e:putting}
\eeq

The deduction above is possible thanks to the trace class property of the operators involved and, again, due to the smoothness of the conformal factors.  Let us denote by
\[ \dot \varphi_0 := \pa_u \varphi_u |_{u=0}. \]
One key point is that the operator ${\cM}_{2 \dot \varphi_0} \big(e^{-t \Delta_{h_0} }  -P_{\Ker(\Delta_{h_0})}\big)$ is trace class and its trace admits an asymptotic expansion as $t \downarrow 0$.
Furthermore, due to the spectral gap of $\Delta_{h_0}$ at zero, the trace of this operator decays as $O(e^{-c/t})$ for some $c>0$ as $t \uparrow \infty$.

Moreover, we also have estimates on the heat kernels demonstrated in \S \ref{hkest}.   Then, to compute the variation of $- \log \det \Delta_g$, one differentiates equation (\ref{e:putting}) with respect to $s$ and sets $s=0$.  The integral with respect to $t$ is split into the corresponding integrals over $[0,1]$ and $[1,\infty)$. The second term yields an analytic function that does not contribute to the Polyakov formula, due to the factor of $s/\Gamma(s)$.  The first term is computed using the short time asymptotic expansion of the trace.

By Corollary \ref{cor:exp1}, as $t \downarrow 0$, there is an expansion of the form
\[ \Tr_{L^2 (M, h_0)} \Big(2 \dot \varphi_0 (z) \big(e^{-t \Delta_{h_0} }  -P_{\Ker(\Delta_{h_0})}\big) \Big) \sim
a_0 t^{-1} + a_1 t^{-\frac 1 2} + a_{2,1} + O(t^{\frac 1 2} \log(t) ).\]
Then, by \eqref{e:putting} we have that
\[ \left . \frac{\pa}{\pa u} \left( - \log \det (\Delta_{h_u})\right) \right|_{u=0}  = a_{2,1}. \]
Now, we use the heat kernel and heat trace calculations to compute this term.  The heat kernel restricted to the diagonal has an expansion for small values of $t$ in powers of $t^{-1/2}$.  For any $z\in M$, at any positive distance from the boundary and the conical singularities, one has
\[ H_0 (t, z, z) \sim \frac{1}{4\pi t} + \frac{\Scal_{h_0} (z)}{24 \pi} + \mathcal O (t^{1/2} \log(t)). \]
Moreover, the Schwartz kernel of $P_{\Ker(\Delta_{h_0})}$ is simply $\frac{1}{A_{h_0}}$, where $A_{h_0}$ denotes the area of $(\bM, h_0)$.  Consequently, in case there are no smooth boundary components, the kernel of $\Delta_{h_0}$ consists of the constant functions, and so in this case the contribution to $a_{2,1}$ from the interior $M$ is:
\[ \int_{M} 2 \dot \varphi_0(z) \left(\frac{\Scal_{h_0}(z)}{24 \pi} -\frac{1}{A_{h_0}}\right) dA_{h_0}. \]

Next, we compute the contribution from the conical singularities using the construction of the heat kernel.  For this we lift the product of the heat kernel restricted to the diagonal and $\dot \varphi_0$ to the single heat space constructed in \cite[\S 3.2]{nrs}.  In this construction, the conical singularities are `blown up', being replaced by boundary faces.  Consequently, since the conformal factor is smooth up to the cone point, on the single heat space $2 \dot \varphi_0$ is constant along each of these blown up faces, being equal to its value at the cone point.  The leading order behavior of the heat kernel at these faces is that of a heat kernel for an exact infinite cone. The contribution to the $t^0$ term in the trace is therefore the product of the corresponding term (the $t^0$ coefficient)  in the expansion of the trace of a truncated exact cone and the value of $2 \dot \varphi_0$ at the cone point.  For a cone angle $\gamma$ at the cone point $p$, this contribution has been computed in \cite[4.42]{cheeger}.  We present an independent calculation of this contribution in Appendix  \ref{s:heatcalcapp}
which coincides with Cheeger's, so that together with the conformal factor we obtain
\[ 2 \dot \varphi_0 (p) \frac{(2\pi)^2 - \gamma^2}{24 \pi \gamma}. \]
Putting all terms together, when there are no smooth boundary components we have the differentiated Polyakov formula:
\begin{multline} \left . \frac{\pa}{\pa u} \left(- \log \det(\Delta_{h_u})\right)\right|_{u=0} =  \int_{M} 2 \dot \varphi_0(z) \left(\frac{\Scal_{h_0}(z)}{24 \pi} -\frac{1}{A_{h_0}}\right) dA_{h_0} \\ + \sum_{i=1} ^m 2 \dot \varphi_0 (p_i) \frac{(2\pi)^2 - \gamma_i^2}{24 \pi \gamma_i}.  \end{multline}
Above, the sum is over the $m$ isolated cone points, $p_1, \ldots, p_m$ with corresponding angles $\gamma_1, \ldots, \gamma_m$.

In order to have an integrated Polyakov formulas we assume that the conformal factors are of the form
\begin{equation} \label{eq:cf-integrated} \varphi_u = \varphi_0 + u \eta, \quad \eta \in \cC^\infty(\ovM) = \cC^\infty(\bM), \end{equation} 
and therefore
\[ h_0 = e^{2\varphi_0} g, \quad \text{ and, } \quad  \dot \varphi_0 = \eta.\]

Recall that the scalar curvatures of metrics $h_0$ and $g$ are related by
\[ \Scal_{h_0}(z) = e^{-2\varphi_0} \big(\Scal_{g}(z) + 2 \Delta_g \varphi_0 \big). \]
Then, we compute in the case where $\bM$ has only cone points, but no smooth boundary components, $\pa \bM = \emptyset$,
\[
 \left. \frac{\partial}{\partial u} \int_{M} \vert \nabla_{g}(\varphi_0 + u \eta) \vert^{2} \ dA_{g} \right\vert _{u=0} = 2 \int_M \big(\nabla_g \eta \big)  \cdot \big(\nabla_g \varphi_0 \big) dA_g\]
 \[ = 2 \int_M \eta \ \Delta_g \varphi_0 dA_g,\]
 by the sign convention chosen for the Laplace operator \eqref{laplacesign}.  Moreover, we also compute
\begin{eqnarray*}
\left.{\frac{\partial}{\partial u}} \int_{M}(\varphi_0+u\eta) \Scal_{g}\ dA_{g}\right\vert _{u=0} &=& \int_{M}  \eta \Scal_{g}  \ dA_{g},\\
\left.{\frac{\partial}{\partial u}} \log(A_{h_u}) \right\vert _{u=0} &=&  \frac{1}{A_{h_0}} \int_{M} 2\eta (z) \ dA_{h_0}.
\end{eqnarray*}
Consequently in this case,
\[ \left . \frac{\pa}{\pa u} \left(- \log \det (\Delta_{h_u}) \right) \right|_{u=0} = \int_{M} 2 \eta \left(\frac{\Scal_{h_0}}{24 \pi} -\frac{1}{A_{h_0}}\right) dA_{h_0}+ \sum_{i=1} ^m 2 \eta (p_i) \frac{(2\pi)^2 - \gamma_i^2}{24 \pi \gamma_i}
\]

\[
= \frac{1}{12 \pi} \int_M \eta (e^{-2 \varphi_0} (\Scal_g  + 2 \Delta_g \varphi_0)) e^{2 \varphi_0} dA_g - \frac{1}{A_{h_0}} \int_M 2 \eta  dA_{h_0} + \sum_{i=1} ^m 2 \eta (p_i) \frac{(2\pi)^2 - \gamma_i^2}{24 \pi \gamma_i}
\]
\[ = \frac{1}{12 \pi} \int_M \eta \Scal_g dA_g + \frac{1}{12 \pi} \int_M 2 \eta \Delta_g \varphi_0 dA_g - \frac{1}{A_{h_0}} \int_M 2 \eta dA_{h_0} + \sum_{i=1} ^m 2 \eta (p_i) \frac{(2\pi)^2 - \gamma_i^2}{24 \pi \gamma_i}\]
\begin{multline*}
= \frac{\pa}{\pa u} \left . \left( \frac{1}{12 \pi}  \int_{M}(\varphi_0+u\eta) \Scal_{g}\ dA_{g} + \frac{1}{12 \pi}  \int_{M} \vert \nabla_{g}(\varphi_0 + u \eta) \vert^{2} \ dA_{g}\right) \right|_{u=0} \\
 + \frac{\pa}{\pa u} \left( \left . -  \log(A_{h_u}) + \sum_{i=1} ^m  2 (\varphi_0 + u \eta) (p_i)  \left(\frac{ (2\pi)^2 - \gamma_i ^2}{24 \pi \gamma_i}\right) \right) \right|_{u=0}.
\end{multline*}
Therefore,
\begin{multline}
- \log \det (\Delta_{h_0})
=  \frac{1}{12 \pi} \int_{M} \Scal_{g}\ \varphi_0 \ dA_{g}
+ \frac{1}{12\pi}  \int_{M} \vert \nabla_{g}\varphi_0 \vert^{2} \ dA_{g} \\
-  \log(A_{h_0})
+  \sum_{i=1} ^m 2 \varphi_0 (p_i) \left(\frac{ (2\pi)^2 - \gamma_i ^2}{24 \pi \gamma_i} \right)  + C.
\end{multline}
The constant is obtained by setting $\varphi_0 = 0$, which gives
\[ - \log \det(\Delta_g) = - \log A_g + C \implies C = - \log \det (\Delta_g) +  \log A_g.\]
Thus
\begin{multline}
 \log \det (\Delta_{g}) - \log \det (\Delta_{h_0})
=  \frac{1}{12 \pi} \int_{M} \Scal_{g}\ \varphi_0 \ dA_{g}
+ \frac{1}{12 \pi}  \int_{M} \vert \nabla_{g}\varphi_0 \vert^{2} \ dA_{g} \\
-  \log(A_{h_0}) +  \log(A_g)
+  \sum_{i=1} ^m 2 \varphi_0 (p_i) \left(\frac{ (2\pi)^2 - \gamma_i ^2}{24 \pi \gamma_i} \right).
\end{multline}
This proves theorem \ref{t:smoothcf-nb}. \\

Next we consider the case in which $\bM$ has smooth boundary components, $\pa \bM\neq 0$. We consider the Dirichlet Laplacian, then the kernel of $\Delta_{h_0}$ is trivial, and so in that case the contribution to $a_{2,1}$ from the interior is:
\[ \int_{M} 2 \dot \varphi_0  \left(\frac{\Scal_{h_0}}{24 \pi}\right) dA_{h_0}. \]
Near smooth boundary components, we use boundary normal coordinates $(x,y)$ such that the boundary is at $y=0$ and $x$ is the coordinate inside the boundary. In these coordinates the metric near a boundary component has the following form:
\[ \wt{\omega}(x,y) dx^2 + dy^2, \]
for some function $\wt{\omega}$ smooth in a neighborhood of the given boundary component and up to this component and $\wt{\omega}(x,0)=1$.  By \cite[Proposition 5.8]{nrs}, the heat kernel in this case, has leading order behavior given by \cite[eqn.(4.11)]{nrs}.  We may therefore use the model of the Dirichlet heat kernel in the Euclidean half-plane $\R^2_+$ to compute the contribution to the short time asymptotic expansion of the boundary components.
 On $\R^2_+$, for the edge at $y=0$ in Cartesian coordinates, the Dirichlet heat kernel  is given by,
\[ H_{\R^2_+}(t,(x,y),(x',y')) = \frac{1}{4\pi t} (e^{-\frac{(x-x')^2 + (y-y')^2}{4t}}  - e^{-\frac{(x - x')^2 + (y + y')^2}{4t}}).\]

Therefore in a small neighborhood of the boundary component, the heat kernel restricted to the diagonal has leading term
\[ \frac{1-e^{-y^2/t}}{4\pi t}. \]
The conformal factor is smooth up to the boundary and therefore has a Taylor expansion in powers of $y$.  So, near a smooth boundary component, for simplicity denoted by $\pa \bM$, we compute for $y_0>0$ small,
\[ \int_{\pa \bM} \int_{y=0} ^{y_0} 2 \dot \varphi_0 (x, y) \frac{1-e^{-y^2/t}}{4\pi t} dy dx. \]
Recall that we are looking for the contribution of the boundary to the constant coefficient $a_{2,1}$ in the asymptotic expansion of the trace of ${\cM}_{2 \dot \varphi_0}  e^{-t \Delta_{h_0} }$. The only possible contribution to this term comes from
\[ - \frac{1}{4\pi t} \int_{\pa \bM} \int_{0} ^{y_0} e^{-y^2/t} 2 \dot \varphi_0 (x,y) dy dx. \]
The Taylor expansion of $ \dot \varphi_0$ near $y=0$ is
\[ \dot \varphi_0 (x,y) \sim b_0 (x) + y b_1 (x) + R(x,y), \quad R(x,y) = \mathcal O (y^2). \]
We compute
\[ -\frac{2}{4\pi t} \int_{\pa \bM} b_0 (x) \int_{0} ^{y_0} e^{-y^2/t} dy dx = - \frac{2}{4\pi \sqrt t} \int_{\pa \bM} b_0 (x) \int_0 ^{y_0/\sqrt t} e^{-u^2} du dx \]
\[= - \frac{2 \sqrt \pi}{8 \pi \sqrt t} \int_{\pa \bM} b_0 (x) dx + \mathcal O(t^\infty). \]
Above we have used the substitution $u = y/\sqrt t$.  Consequently the first term in the Taylor expansion does not contribute to $a_{2,1}$. Next we compute
\begin{multline*}
- \frac{2}{4\pi t} \int_{\pa \bM} b_1 (x) \int_0 ^{y_0} y e^{-y^2/t} dy dx = \frac{-1}{4\pi} \int_{\pa \bM} b_1 (x) \int_0 ^{y_0 ^2/t} e^{-u} du\\ = \frac{1}{4\pi} \int_{\pa \bM} b_1 (x) dx + \mathcal O(t^\infty),\end{multline*}
where
\[ b_1 (x) = \frac{\pa \dot \varphi_0}{\pa y} (x, 0), \]
is the normal derivative of $\dot \varphi_0$ along the smooth boundary component. For the remainder term, a similar computation shows that it is $\mathcal O (\sqrt t)$ as $t \downarrow 0$. Consequently, switching from local coordinates to global coordinates at the boundary, the contribution from the smooth boundary components to the constant term $a_{2,1}$ is
\[ \frac{1}{4\pi} \int_{\pa \bM} \frac{\pa \dot \varphi_0}{\pa n_{h_0}} (x, 0) dx_{h_0}. \]
The notation $dx_{h_0}$ indicates that the integration along the boundary is with respect to the boundary measure induced by the Riemannian metric $h_0$, and similarly, the normal derivative along the boundary is with respect to the boundary measure induced by $h_0$.

There is also a contribution from the smooth boundary components due to the local expansion of the heat kernel for short times in terms of the curvatures.
The term in the expansion of the kernel is well known to be
\[ \frac{1}{12\pi}  k_{h_0} (x),\]
\noindent with $k_{h_0}$ being the geodesic curvature of the boundary with respect to the Riemannian metric $h_0$, \cite{ms}. Thus, recalling the conformal factor which is smooth at all boundary components, this gives a contribution:
\[ \frac{1}{12\pi} \int_{\pa \bM} 2 \dot \varphi_0 (x) k_{h_0} (x) dx_{h_0}. \]
Thus the total contribution from the boundary is
\[ \frac{1}{4\pi} \int_{\pa \bM} \frac{ \pa \dot \varphi_0}{\pa n_{h_0}} (x) dx_{h_0}+\frac{1}{12\pi} \int_{\pa \bM} 2 \dot \varphi_0 (x) k_{h_0} (x) dx_{h_0}. \]

Therefore, for a surface with conical singularities and at least one smooth boundary component, we obtain the variational Polyakov formula
\begin{multline} \left . \frac{\pa}{\pa u} \left(- \log \det(\Delta_{h_u}) \right)\right|_{u=0} =  \int_{M} 2 \dot \varphi_0(z) \frac{\Scal_{h_0}(z)}{24 \pi} dA_{h_0} \\ +\frac{1}{4\pi} \int_{\pa \bM} \frac{\pa \dot \varphi_0}{\pa n_{h_0}} (x) dx_{h_0}+\frac{1}{12\pi} \int_{\pa \bM} 2 \dot \varphi_0 (x) k_{h_0} (x) dx_{h_0}\\ + \sum_{i=1} ^m 2 \dot \varphi_0 (p_i) \frac{(2\pi)^2 - \gamma_i^2}{24 \pi \gamma_i}. \label{e:dpfcsasbc} \end{multline}

To obtain the integrated Polyakov formula, we assume that the conformal factors are given by \eqref{eq:cf-integrated}, and we note that the relationship between the corresponding curvatures is
\[ \Scal_{h_0} = e^{-2 \varphi_0} \left( \Scal_g + 2 \Delta_g \varphi_0\right), \quad k_{h_0} = e^{-\varphi_0} \left(k_{g} +  \frac{\partial \varphi_0}{\partial_{n_g}} \right).\]
We further have
\[ dx_{h_0} = e^{\varphi_0} dx_g, \quad \frac{\pa}{\pa n_{h_0}} = e^{-\varphi_0} \frac{\pa}{\pa n_g}.\]
Here $n_g$ denotes the unitary (with respect to $g$) outer normal vector at $\pa \bM$. With this, equation \eqref{e:dpfcsasbc} becomes
\beq \left . \frac{\pa}{\pa u} \left( -  \log \det(\Delta_{h_u}) \right) \right|_{u=0} &=&  \frac{1}{12\pi} \int_M \eta \Scal_g dA_g + \frac{1}{12\pi} \int_M 2 \eta \Delta_g \varphi_0 dA_g \nn \\
&+&\frac{1}{4\pi} \int_{\pa \bM} \frac{\pa \eta}{\pa n_g} dx_g +\frac{1}{12\pi} \int_{\pa \bM} 2 \eta \left( k_g + \frac{\pa \varphi_0}{\pa n_g} \right) dx_g \nn \\ 
&+& \sum_{i=1} ^m 2 \eta (p_i) \frac{(2\pi)^2 - \gamma_i^2}{24 \pi \gamma_i}.  \nn \eeq
Now we proceed in the same way as we did in the proof of the integrated formula for surfaces with conical singularities without smooth boundary above, but taking into account the terms coming from the boundary.   By the sign convention chosen for the Laplace operator \eqref{laplacesign} we have
\begin{multline*}
 \left. \frac{\partial}{\partial u} \int_{M} \vert \nabla_{g}(\varphi_0 + u \eta) \vert^{2} \ dA_{g} \right\vert _{u=0} = 2 \int_M \nabla_g \eta \cdot \nabla_g \varphi_0 dA_g\\
  = 2 \int_M \eta \Delta_g \varphi_0 dA_g + 2 \int_{\pa \bM} \eta \frac{\pa \varphi_0}{\pa n_g}(x) dx_g.\end{multline*}
  Since
 \[ \frac{\pa}{\pa u} \left . 2 \int_{\pa \bM} (\varphi_0 + u \eta) \frac{\pa \varphi_0}{\pa n_g} dx_g \right|_{u=0} = 2 \int_{\pa \bM} \eta \frac{\pa \varphi_0}{\pa n_g}(x) dx_g,\]
 we obtain
 \[  \frac{\partial}{\partial u}  \left . \left(\int_{M} \vert \nabla_{g}(\varphi_0 + u \eta) \vert^{2} \ dA_{g} -  2 \int_{\pa \bM} (\varphi_0 + u \eta) \frac{\pa \varphi_0}{\pa n_g} dx_g \right) \right|_{u=0} = 2 \int_M \eta \Delta_g \varphi_0 dA_g.\]

Using $\dot \varphi_0 = \eta$, we compute that
\[ \frac{1}{4\pi} \int_{\pa \bM} \frac{\pa \dot \varphi_0}{\pa n_{h_0}} (x) dx_{h_0} = \frac{1}{4\pi} \int_{\pa \bM} e^{-\varphi_0} \frac{\pa \dot \varphi_0}{\pa n_g} e^{\varphi_0} dx_g = \frac{1}{4\pi} \int_{\pa \bM} \frac{\pa \dot \varphi_0}{\pa n_g} dx_g = \frac{1}{4\pi} \int_{\pa \bM} \frac{\pa \eta}{\pa n_g} dx_g.\]
Hence
\[ \frac{\pa }{\pa u} \left. \frac{1}{4\pi} \int_{\pa \bM} \frac{\pa}{\pa n_g} (\varphi_0 + u \eta) dx_g \right|_{u=0} = \frac{1}{4\pi} \int_{\pa \bM} \frac{\pa \eta}{\pa n_g} dx_g.\]
In addition, we also have
\begin{multline*} \frac{\pa}{\pa u} \left . \frac{1}{12\pi} \int_{\pa \bM} 2 (\varphi_0 + u \eta) k_{h_0} dx_{h_0} \right|_{u=0} = \frac{1}{12\pi} \int_{M} 2 \eta k_{h_0} dx_{h_0} \\
= \frac{1}{12\pi} \int_{\pa \bM} 2 \eta e^{-\varphi_0} \left( k_g + \frac{\pa \varphi_0}{\pa n_g} \right) e^{\varphi_0} dx_g =  \frac{1}{12\pi} \int_{\pa \bM} 2 \eta \left( k_g + \frac{\pa \varphi_0}{\pa n_g} \right) dx_g.\end{multline*}

Combining all these equations together, we obtain

\begin{equation*}
\begin{split}
& \left . \frac{\pa}{\pa u} \left( -  \log \det(\Delta_{h_u}) \right) \right|_{u=0} = \left . \frac{\pa}{\pa u} \left( \frac{1}{12\pi}  \int_{M}(\varphi_0+u\eta) \Scal_{g} dA_{g} \right) \right|_{u=0} \\
&\phantom{ \log \det }
+ \frac{\pa}{\pa u} \left . \frac{1}{12\pi} \left(\int_{M} \vert \nabla_{g}(\varphi_0 + u \eta) \vert^{2} \ dA_{g} -  2 \int_{\pa \bM} (\varphi_0 + u \eta) \frac{\pa \varphi_0}{\pa n_g} dx_g \right) \right|_{u=0} \\
&\phantom{\log \det} + \frac{\pa }{\pa u} \left. \frac{1}{4\pi} \int_{\pa \bM} \frac{\pa}{\pa n_g} (\varphi_0 + u \eta) dx_g \right|_{u=0}  +  \frac{\pa}{\pa u} \left . \frac{1}{12\pi} \int_{\pa \bM} 2 (\varphi_0 + u \eta) \left( k_g + \frac{\pa \varphi_0}{\pa n_g} \right) dx_g\right|_{u=0} \\
&\phantom{\log \det } + \frac{\pa}{\pa u} \left .  \sum_{i=1} ^m 2 (\varphi_0 + u \eta) (p_i) \frac{(2\pi)^2 - \gamma_i^2}{24 \pi \gamma_i}\right|_{u=0}
 \end{split}
 \end{equation*}

\begin{multline*} = \left . \frac{\pa}{\pa u} \left( \frac{1}{12\pi}  \int_{M}(\varphi_0+u\eta) \Scal_{g} dA_{g} + \frac{1}{12\pi} \int_{M} \vert \nabla_{g}(\varphi_0 + u \eta) \vert^{2} \ dA_{g} \right) \right|_{u=0} \\
+ \frac{\pa}{\pa u} \left . \left(  \frac{1}{4\pi} \int_{\pa \bM} \frac{\pa}{\pa n_g} (\varphi_0 + u \eta) dx_g + \frac{1}{6\pi} \int_{\pa \bM} (\varphi_0 + u \eta) k_g  dx_g \right)\right|_{u=0} \\
 + \frac{\pa}{\pa u} \left .  \sum_{i=1} ^m 2 (\varphi_0 + u \eta) (p_i) \frac{(2\pi)^2 - \gamma_i^2}{24 \pi \gamma_i}\right|_{u=0} \end{multline*}

We therefore obtain
\begin{multline*}
 \log \det (\Delta_{h_0}) = -\frac{1}{12 \pi}  \int_{M} \varphi_0 \Scal_{g} dA_{g} -  \frac{1}{12\pi} \int_{M} \vert \nabla_{g} \varphi_0  \vert^{2} \ dA_{g}\\ -  \frac{1}{4\pi} \int_{\pa \bM} \frac{\pa \varphi_0}{\pa n_g}dx_g
  - \frac{1}{6\pi} \int_{\pa \bM}  \varphi_0 k_{g} dx_g - \sum_{i=1} ^m 2 \varphi_0 (p_i) \frac{(2\pi)^2 - \gamma_i^2}{24 \pi \gamma_i} + C.\end{multline*}
To determine the constant, we set $\varphi_0=0$, which gives
\[C = \log \det(\Delta_g).\]
Consequently, we obtain the formula relating the determinants
\begin{multline*}
\log \det (\Delta_{h_0}) - \log \det(\Delta_g)
 = -\frac{1}{12 \pi}  \int_{M} \varphi_0 \Scal_{g} dA_{g} -  \frac{1}{12\pi} \int_{M} \vert \nabla_{g} \varphi_0  \vert^{2} \ dA_{g} \\
 -  \frac{1}{4\pi} \int_{\pa \bM} \frac{\pa \varphi_0}{\pa n_g}dx_g
 - \frac{1}{6\pi} \int_{\pa \bM}  \varphi_0 k_{g} dx_g - \sum_{i=1} ^m 2 \varphi_0 (p_i) \frac{(2\pi)^2 - \gamma_i^2}{24 \pi \gamma_i}.
 \end{multline*}
This completes the proof of Theorem \ref{t:smoothcf-b}.  To complete the proof of Theorem \ref{thm:domains}, we note that the contribution from the corners, that is the sum, simply has a change from $(2\pi)^2$ in the numerator to $\pi^2$.
\qed


\section{The determinant on finite circular sectors and finite cones}  \label{s:explicit}
We use \cite{bord96} to compute the zeta regularized determinant of the Laplacian with Dirichlet boundary condition on circular sectors.  The eigenvalues of the sector $S_\alpha$ of opening angle $\alpha$ and radius $1$ with the Dirichlet boundary condition are 
\[ \lambda_{n,\ell}^2 = \textrm{ the square of the $n^{th}$ positive root of the Bessel function, $J_{\nu_\ell}$,} \]
with $\nu_\ell = \ell \pi /\alpha$.
This allows one to write the associated zeta function in terms of a contour integral as
\[  \zeta_{S_\alpha} (s) = \sum_{\ell =1}^\infty \sum_{n=1}^\infty \lambda_{n,\ell} ^{-2s}  = \sum_{\ell =1}^\infty \frac 1 {2\pi i} \int\limits_\gamma k^{-2s} \frac \partial {\partial k} \log J_{\nu_\ell} (k) dk, \]
where the contour $\gamma$ encloses all the (positive) eigenvalues in the right half-plane. This representation is valid for $\Rea s > 1$, but the analytical continuation has been constructed in detail (in greater generality) in \cite{bord96} (see also \cite{kkbook}), and we therefore omit the details here.  Subtracting and adding the suitable uniform asymptotics of the Bessel function,
in the notation of \cite[(3.3)]{bord96} this construction leads to the base zeta function and the Barnes type zeta function \cite[(9.5)]{bord96}, that are respectively,
\begin{equation} \label{sec4}
\zeta_{\cN} (s) = \sum_{\ell =1}^\infty \nu_\ell ^{-2s} = \left( \frac \pi \alpha\right)^{-2s} \zeta_R (2s), \quad
\zeta_{{\cN} +1} (z) = \sum_{n=1}^\infty \sum_{\ell =1}^\infty \left( \frac{\pi \ell} \alpha +n\right)^{-z}.
\end{equation}
We shall first use these zeta functions to compute $- \log \det (\Delta_{S_\alpha})$.
To do this, we shall also require  the polynomial $D_1 (t) = \frac 1 8 t - \frac 5 {24} t^3$.
With these quantities, we have \cite[(9.8)]{bord96}
\begin{equation} \label{sec6}
- \log \det (\Delta_{S_\alpha}) = \zeta ' _{S_\alpha} (0) = \zeta_{{\cN} +1} ' (0) + \log 2 \left( \zeta_{\cN} \left( - \frac 1 2 \right) + 2 \mbox{Res } \zeta_{\cN} \left( \frac 1 2 \right) D_1 (1) \right) \end{equation}
\[ + 2 \mbox{Res } \zeta_{\cN} \left( \frac 1 2\right) \,\, \int\limits_0^1 \frac{ D_1 (t) - t D_1 (1)} {t (1-t^2)} dt .\]
All quantities except $\zeta_{{\cN}+1} ' (0)$ are easily computed.  In particular, one finds
\beq
& & \zeta_{\cN} \left( - \frac 1 2 \right) = \frac \pi \alpha \zeta_R (-1) = - \frac \pi {12\alpha} , \quad \quad \mbox{Res } \zeta_{\cN} \left( \frac 1 2 \right) =
\frac \alpha {2\pi} \mbox{Res } \zeta _R (1) = \frac \alpha {2\pi}, \nn\\
& &D_1 (1) = - \frac 1 {12}, \quad \quad \int\limits_0^1 \frac{ D_1 (t) - t D_1 (1) } {t (1-t^2)} dt = \frac 5 {24},\nn
\eeq
and so
\beq
\zeta ' _{S_\alpha} (0) = \zeta_{{\cN} +1} ' (0) + \frac 5 {24} \frac \alpha \pi -\frac 1 {12} \log 2 \left( \frac \alpha {\pi} + \frac \pi { \alpha} \right) .\label{sec7}
\eeq
In order to compare our results with the literature, we rewrite $\zeta_{\cN +1} (z)$ in terms of the commonly used Barnes zeta function
\beq
\zeta_{\mathcal B} (z; a,b,x) = \sum_{m,n=0}^\infty \left( am+bn+x\right)^{-z}.\nn
\eeq 
Straightforward rewriting shows
\beq
\zeta_{\cN +1} (z) = \left( \frac \alpha \pi \right)^z \left( \zeta _{\mathcal B} \left( z; \frac \alpha \pi , 1,1\right) - \zeta_R (z) \right) \nn
\eeq
and thus
\beq
\zeta_{\cN +1} ' (0) &=& \log \left( \frac \alpha \pi \right) \left[ \zeta_{\mathcal B} \left( 0; \frac \alpha \pi , 1,1\right) - \zeta _R (0) \right]+ \zeta _{\mathcal B} ' \left( 0; \frac \alpha \pi ,1,1\right) - 
\zeta_R ' (0) \nn\\
&=& \log \left( \frac \alpha \pi \right) \left[ \frac 1 4 + \frac 1 {12} \left( \frac \alpha \pi + \frac \pi \alpha \right) \right] + \zeta _{\mathcal B} ' \left( 0; \frac \alpha \pi , 1,1\right) + \frac 1 2 \log (2\pi ).\nn
\eeq
This shows, rewriting eq.~(\ref{sec7}), that 
\beq
\zeta ' _{S_\alpha} (0) &=& \frac 1 {12} \log \left( \frac \alpha {2\pi} \right) \left[ \frac \alpha \pi + \frac \pi \alpha \right] + \frac 1 4 \log \left( \frac \alpha \pi \right) \nn \\ 
&+& \frac 1 2 \log (2\pi ) 
+ \frac 5 {24 } \frac \alpha \pi +\zeta_{\mathcal B} ' \left( 0; \frac \alpha \pi , 1,1\right). \label{eq:tsectors_pf1}
\eeq

Similarly, for a cone of opening angle $2\alpha$ with Dirichlet boundary conditions imposed at $r=1$, where $(r,\theta )$ are coordinates on the cone whose tip lies at $r=0$, the associated zeta function is
(with notation as before)
\beq
\zeta_{C_{2\alpha}} (s) &=& \sum_{\ell =-\infty}^\infty \sum_{n=1}^\infty \lambda_{n,|\ell |} ^{-2s} = 2 \sum_{\ell =1} ^\infty \sum_{n=1}^\infty \lambda_{n,\ell } ^{-2s} + \sum_{n=1}^\infty \lambda_{n,0}^{-2s}\nn\\
 &=& 2 \zeta_{S_\alpha} (s) + \sum_{n=1}^\infty \lambda_{n,0}^{-2s}.\nn
\eeq
Studies of the zeta function associated with the zeroes of the Bessel function (with no sum over $\ell$ performed) seem to go back to Hawkins \cite{hawk} with several more details provided in 
\cite{actben}. In particular, in \cite[(2.27)]{actben} (see eq. (1.1) for the definition of the relevant 'Bessel $\zeta$-function' containing a factor of $\pi$), the derivative at zero of the last term above
has been computed, which allows us to conclude
\beq
\zeta _{C_{2\alpha } }' (0) = 2 \zeta _{S_\alpha } ' (0) - \frac 1 2 \log (2\pi ). \nn
\eeq
Setting $\alpha = \pi a$, this gives
\beq
\zeta _{C_{2\pi a} } ' (0) &=& \nn \\ 
 && 2 \zeta _{\mathcal B} ' (0; a,1,1 ) + \frac 1 6 \log \left( \frac a 2 \right) \left( a + \frac 1 a \right) + \frac 1 2 \log a + \frac 1 2 \log (2\pi ) + \frac 5 {12 } a, \label{eq:zeta_c2pi_a} 
\eeq
a result due to Spreafico \cite{Spre,Spre1}; for a detailed explanation see \cite{klev,kalvin3}. 

\begin{proof}[Proof of Theorem \ref{t:sectors}]
Equation \ref{eq:magicvar1} is obtained using \eqref{eq:tsectors_pf1} together with the representation of the Barnes zeta function as given in \cite[p.~30]{kalvin2}; see also \cite{AuSa}.  The variational formula given in \eqref{eq:magic_angle} is then obtained by observing that the integrals converge absolutely and may be differentiated with respect to the parameter $\alpha$.  To be more rigorous one may use the dominated convergence theorem.  
To obtain the remaining three variational formulas in this theorem, we recall \cite[Theorem 4]{AldRow}.\footnote{There is unfortunately a typo in the statement of the theorem \cite[Theorem 4]{AldRow} that has been corrected here.  We note that this is mere a transcription error from the contents of the proof.  Further, the boundary contribution was overlooked in \cite{AldRow}, but has been corrected in  \cite{AldRowE}.}

\begin{theorem}[Theorem 4 in \cite{AldRow}] \label{th:correctformula} The angular variation for a sector of opening angle $\alpha$ with radius one is
\[  \frac{d}{d\alpha} (- \log \det (\Delta_{S_\alpha})) = \frac{1}{3\pi} + \frac{\pi}{12 \alpha^2} +  \sum_{k\in W_{\alpha}} \frac{-2\gamma_e + \log 2 - \log\left({1-\cos(2k\alpha)}\right) }{4 \pi (1-\cos(2k\alpha))} \]
\[ + \frac{2}{\alpha} \sin(\pi^2/\alpha) \int_\R \frac{ - \log 2 + 2 \gamma_e + \log(1+\cosh(s))}{16 \pi (1+\cosh(s)) (\cosh(\pi s/\alpha) - \cos(\pi^2/\alpha))} ds. \]
Above, we have defined
\[ k_{min} = \ceil*{ \frac{-\pi}{2\alpha} }, \textrm{ and } k_{max} = \floor*{\frac{\pi}{2\alpha}} \textrm{ if } \frac{\pi}{2\alpha} \not\in \Z, \textrm{ otherwise } k_{max} = \frac{\pi}{2\alpha} - 1, \]
and
\begin{equation} \label{eq:w_alpha} W_{\alpha} =\left\{  k \in \left( \Z \bigcap \left[k_{min},  k_{max}\right]\right) \setminus \left\{ \frac{\ell\pi}{\alpha} \right\}_{\ell \in \Z} \right\}. \end{equation}
\end{theorem}

The purpose of the following considerations is to simplify the expressions in Theorem \ref{th:correctformula} such as to match those in Theorem \ref{t:sectors}.  
\begin{prop} \label{prop:w_alpha}
For any angle $\alpha \in (0, \pi)$ the set $W_\alpha$ is precisely the set of integers
\[\Z \supset W_\alpha = \{ j \}_{k_{min}} ^{k_{max}} \setminus \{ 0 \}. \]
Moreover, for all $\alpha$
\[ k_{max} = \ceil*{ \frac{\pi}{2\alpha} -1}.\]
\end{prop}
\begin{proof} The set $W_\alpha$ as in \eqref{eq:w_alpha} excludes integers of the form $k = \frac{\ell \pi}{\alpha}$, for $\ell \in \Z$.  Consequently zero is always excluded from $W_\alpha$ for any $\alpha$.  Assume now that there is a non-zero integer $k = \frac{\ell \pi}{\alpha}$.  Then it follows immediately that
\[ \alpha = \frac{\ell \pi}{k}.\]
Without loss of generality, assume $\ell$ and $k$ are positive.  Since $\ell \geq 1$,
\[ - k < - \frac{k}{2 \ell} \leq k_{min}. \]
If $\frac{\pi}{2\alpha} \not \in \Z$, then
\[ k_{max} =  \floor*{\frac{\pi}{2\alpha}} =  \floor*{\frac{k}{2\ell}} < k.\]
If $\alpha = \frac \pi j$ in case $j=2n$ is even, then $k_{min} = -n$, and $k_{max} = n-1$.  In this case, for any integer $\ell$,
\[ \frac{\ell \pi}{\alpha} = 2 n \ell > k_{max}, \quad \forall \ell \in \N.\]
If $\alpha = \frac \pi j$ in case $j=2n+1$ is odd, then $k_{min} = -n$, and $k_{max} = n$.  In this case, for any integer $\ell$,
\[ \frac{\ell \pi}{\alpha} = 2 (n+1) \ell > k_{max}, \quad \forall \ell \in \N.\]
We therefore see that in all cases,
\[ \Z \ni k  = \frac{\ell \pi}{\alpha} \in [k_{min}, k_{max}] \implies k=\ell = 0. \]
Finally, we note that if
\[ \frac{\pi}{2\alpha} \not \in \Z \implies \ceil*{ \frac{\pi}{2\alpha} -1 } = \floor*{ \frac{\pi}{2\alpha}},\]
whereas if
\[ \frac{\pi}{2\alpha} \in \Z \implies k_{max} = \frac{\pi}{2\alpha} - 1 =  \ceil*{ \frac{\pi}{2\alpha} -1 }.\]
\end{proof}

We next manipulate the sums in the variational formula of Theorem \ref{th:correctformula}.  
\begin{prop} \label{prop:reformulate_sum_w_alpha}
For all $\alpha \in (0, \pi)$
\[\sum_{k\in W_{\alpha}} \frac{-2\gamma_e + \log 2  - \log\left({1-\cos(2k\alpha)}\right) }{4 \pi (1-\cos(2k\alpha))} =  -  \sum_{k\in W_{\alpha}} \frac{\gamma_e + \log |\sin(k \alpha)|}{4\pi \sin^2 (k \alpha)} . \]
If $\alpha$ is not of the form $\frac{\pi}{2n}$ for any $n \in \N$, then 
\[ -  \sum_{k\in W_{\alpha}} \frac{\gamma_e + \log |\sin(k \alpha)|}{4\pi \sin^2 (k \alpha)} = -  \sum_{k=1} ^{\ceil*{ \frac{\pi}{2\alpha} -1}}  \frac{\gamma_e + \log |\sin(k \alpha)|}{2\pi \sin^2 (k \alpha)}.\]
In case $\alpha = \frac \pi j$ for some $j \in \N$, then 
\[-  \sum_{k\in W_{\alpha}} \frac{\gamma_e + \log |\sin(k \alpha)|}{4\pi \sin^2 (k \alpha)}  = - \frac{\gamma_e}{12\pi} \left( \frac{\pi^2}{\alpha^2} - 1 \right) - \frac{1}{2\pi} \sum_{k=1} ^{\ceil*{ \frac{\pi}{2\alpha} -1}} \frac{ \log | \sin(k \alpha)|}{\sin^2(k \alpha)}. \]
\end{prop}

\begin{proof}
Note that
\[ \cos(2k \alpha) = \cos^2 (k\alpha) - \sin^2(k \alpha) = 1 - 2 \sin^2 (k \alpha) \implies  1- \cos(2k\alpha) = 2 \sin^2 (k \alpha), \]
and
\[ \log(1-\cos(2k\alpha)) = \log(2 \sin^2 (k \alpha)) = \log 2 + 2 \log |\sin(k\alpha)|. \]
We therefore have
\[ \sum_{k\in W_{\alpha}} \frac{-2\gamma_e + \log 2 - \log\left({1-\cos(2k\alpha)}\right) }{4 \pi (1-\cos(2k\alpha))} \]
\[ = \sum_{k\in W_{\alpha}} \frac{-2\gamma_e - 2\log |\sin(k \alpha)|}{8\pi \sin^2 (k \alpha)} =  -  \sum_{k\in W_{\alpha}} \frac{\gamma_e + \log |\sin(k \alpha)|}{4\pi \sin^2 (k \alpha)}. \]
By the preceding proposition this sum is
\[ \sum_{k=k_{min}} ^{-1}  \frac{\gamma_e + \log |\sin(k \alpha)|}{4\pi \sin^2 (k \alpha)} - \sum_{k=1} ^{k_{max}}  \frac{\gamma_e + \log |\sin(k \alpha)|}{4\pi \sin^2 (k \alpha)}.\]
We note that
\[ \lceil \frac{\pi}{2\alpha} - 1 \rceil = k_{max} = \begin{cases} - k_{min}, & \alpha \neq \frac{\pi}{2n}, \\ - k_{min} - 1, & \alpha = \frac{\pi}{2n}, \end{cases}, \]
where above $n \in \N$.  Next, we observe that
\[ | \sin(k \alpha)| = |\sin(-k \alpha)|, \quad \sin^2 (k \alpha) = \sin^2 (-k \alpha).\]
Consequently
\[ - \sum_{k=k_{min}} ^{-1}  \frac{\gamma_e + \log |\sin(k \alpha)|}{4\pi \sin^2 (k \alpha)} - \sum_{k=1} ^{k_{max}}  \frac{\gamma_e + \log |\sin(k \alpha)|}{4\pi \sin^2 (k \alpha)} \]
\[ = - \sum_{k=1} ^{k_{max}}  \frac{\gamma_e + \log |\sin(k \alpha)|}{2 \pi \sin^2 (k \alpha)} =   - \sum_{k=1} ^{\ceil*{ \frac{\pi}{2\alpha} -1}}  \frac{\gamma_e + \log |\sin(k \alpha)|}{2 \pi \sin^2 (k \alpha)}, \quad \alpha \neq \frac{\pi}{2n}. \]

In case $\alpha = \frac \pi j$ we now proceed to simplify the formula.  For notational convenience, let us define
\[ -n := k_{min} = \ceil*{ - \frac{\pi}{2 \alpha} }, \quad \kmax := \begin{cases} n, & j = 2n+1 \textrm{ is odd,} \\ n-1 & j = 2n, \textrm{ is even.} \end{cases} \]
We therefore consider
\[ - \sum_{k=-n} ^{-1}  \frac{\gamma_e + \log |\sin(k \alpha)|}{4\pi \sin^2 (k \alpha)} -  \sum_{k=1} ^{\kmax}  \frac{\gamma_e + \log |\sin(k \alpha)|}{4\pi \sin^2 (k \alpha)}.\]
We would like to use the fact that $|\sin(x)|$ and $\sin^2(x)$ are periodic with period $\pi$ to shift the first sum.  To do this we observe that in case $j=2n$ is even, we have
\[ \sin^2(k \alpha) = \sin^2((k+2n) \alpha), \quad \log | \sin(k\alpha)| = \log |\sin((k+2n)\alpha)|, \quad \kmax = n-1.\]
We therefore have in this case
\[  - \sum_{k=-n} ^{-1}  \frac{\gamma_e + \log |\sin(k \alpha)|}{4\pi \sin^2 (k \alpha)} = -  \sum_{k=\kmax+1} ^{\kmax+|k_{min}|}  \frac{\gamma_e + \log |\sin(k \alpha)|}{4\pi \sin^2 (k \alpha)} \]
\[ \implies  - \sum_{k=-n} ^{-1}  \frac{\gamma_e + \log |\sin(k \alpha)|}{4\pi \sin^2 (k \alpha)} -  \sum_{k=1} ^{\kmax}  \frac{\gamma_e + \log |\sin(k \alpha)|}{4\pi \sin^2 (k \alpha)} = - \sum_{k=1} ^{\kmax + |k_{min}|}  \frac{\gamma_e + \log |\sin(k \alpha)|}{4\pi \sin^2 (k \alpha)}.\]
Next we consider the case in which $j = 2n+1$ is odd.  Then we have
\[ \sin^2(k \alpha) = \sin^2((k+2n+1) \alpha), \quad \log | \sin(k\alpha)| = \log |\sin((k+2n+1)\alpha)|, \quad \kmax = n.\]
We therefore obtain as well in this case
\[ - \sum_{k=-n} ^{-1}  \frac{\gamma_e + \log |\sin(k \alpha)|}{4\pi \sin^2 (k \alpha)} = -  \sum_{k=\kmax+1} ^{\kmax+|k_{min}|}  \frac{\gamma_e + \log |\sin(k \alpha)|}{4\pi \sin^2 (k \alpha)}.\]
Consequently, in all cases where $\alpha = \frac \pi j$ for some $j \in \N$ we obtain that the sum is
\[  -  \sum_{k=1} ^{\kmax+|k_{min}|}  \frac{\gamma_e + \log |\sin(k \alpha)|}{4\pi \sin^2 (k \alpha)}.\]

Let us define
\begin{equation} \label{eq:cS_and_N} \cS := - \frac 1 4 \sum_{k=1} ^{N} \frac{1}{\sin^2 (k \alpha)}, \quad N := \kmax + |k_{min}| = \begin{cases} 2n-1, & \alpha = \frac{\pi}{2n}, \\ 2n, & \alpha = \frac{\pi}{2n+1}. \end{cases} \end{equation}
We compute that
\[ \sin^2 (k \alpha) = - \frac 1 4 \left( e^{ik\alpha} - e^{-ik\alpha} \right)^2 = - \frac 1 4 e^{2ik\alpha} \left( 1 - e^{-2ik\alpha}\right)^2.\]
Then
\[ \cS =  \sum_{k=1} ^{N} \frac{e^{-2ik\alpha}}{(1-e^{-2ik\alpha})^2} =  \lim_{\eps \to 0} \sum_{k=1} ^{N} \frac{e^{-i(2k\alpha-i \eps)}}{\left( 1 - e^{-i(2k\alpha - i \eps)} \right)^2}.\]
We do this to exploit the geometric series for $|z| < 1$,
\[ \frac{1}{(1-z)^2} = \sum_{\ell = 1} ^\infty \ell z^{\ell - 1}. \]
With $z = e^{-i(2k\alpha - i \eps)}$ we therefore have
\[ \cS =  \lim_{\eps \to 0} \sum_{k=1} ^{N}  e^{-i(2k\alpha - i \eps)} \sum_{\ell = 1} ^\infty \ell e^{-i(2k\alpha - i \eps)(\ell - 1)} =  \lim_{\eps \to 0} \sum_{\ell=1} ^\infty \ell e^{-\eps \ell} \sum_{k=1} ^{N} e^{-i2k\ell\alpha}. \]
The sum over $k$ can be computed using a geometric series.  There are two cases.  In case
\[ 2i\ell \alpha  = 2m\pi i \quad \textrm{with} \quad m \in \N \implies \sum_{k=1} ^{N} e^{-i2k\ell\alpha} = N. \]
This is the case when
\[ \ell = \frac{m\pi}{\alpha}, \quad m \in \N. \]
When $\ell$ is not of this type, the series sums to
\[  \sum_{k=1} ^{N} e^{-i2k\ell\alpha} = \frac{e^{-i 2 \ell \alpha} - e^{-2i\ell \alpha(N +1)} }{1 - e^{-i 2 \ell \alpha} } = -1. \]
The simplification above is due to the fact that by \eqref{eq:cS_and_N}, we always have
\begin{equation} \label{eq:Nplusone}  2 \alpha (N+1) = 2 \pi. \end{equation}

We therefore have
\[ \cS = \lim_{\eps \to 0} \left\{ -  \sum_{\ell\geq 1, \ell \neq \frac{m \pi}{\alpha}, m \in \N}  \ell e^{-\eps \ell}  +  \sum_{\ell= \frac{m \pi}{\alpha}, m \in \N} ^\infty  \ell e^{-\eps \ell} N \right\} .\]
\[ = \lim_{\eps \to 0} \left\{ -  \sum_{\ell\geq 1, \ell \neq \frac{m \pi}{\alpha}, m \in \N}  \ell e^{-\eps \ell}  +  \sum_{\ell= \frac{m \pi}{\alpha}, m \in \N} ^\infty  \ell e^{-\eps \ell} (N+1-1) \right\} .\]
We do this to write
\[ \cS =  \lim_{\eps \to 0} \left\{ -  \sum_{\ell\geq 1}  \ell e^{-\eps \ell}  +  \sum_{\ell= \frac{m \pi}{\alpha}, m \in \N} ^\infty  \ell e^{-\eps \ell} (N+1) \right\} .\]
Next we re-write the second sum in terms of $m$, so that
\[ \cS = \lim_{\eps \to 0} \left\{ -  \sum_{\ell\geq 1}  \ell e^{-\eps \ell}  +  (N+1) \frac \pi \alpha \sum_{ m =1} ^\infty  m e^{-\eps m \pi/\alpha}  \right\}.\]
Here we recall that
\begin{equation} \label{eq:herewerecall} \sum_{\ell = 1} ^\infty \ell e^{-\eps x \ell} = - \frac 1 \eps \frac{d}{dx} \sum_{\ell = 0} ^\infty e^{-\eps x \ell} = \frac{e^{-\eps x}}{(1-e^{-\eps x})^2} = \frac{1}{\eps^2} \frac{1}{x^2} - \frac{1}{12} + \mathcal O ((\eps x)^2). \end{equation}
Consequently
\[ \cS =  \lim_{\eps \to 0} \left \{ - \frac{e^{-\eps }}{(1-e^{-\eps })^2} +  (N+1) \frac \pi \alpha  \frac{e^{-\eps \pi/\alpha}}{(1-e^{-\eps \pi/\alpha})^2} \right\}.\]
By \eqref{eq:Nplusone},
\[ (N+1) \frac \pi \alpha = \frac{\pi^2}{\alpha^2}. \]
Therefore  by \eqref{eq:herewerecall},
\[ \cS =  \lim_{\eps \to 0} \left \{ - \frac{1}{\eps^2} + \frac{1}{12} + \mathcal O(\eps) + \frac{\pi^2}{\alpha^2} \left( \frac{1}{\eps^2} \frac{\alpha^2}{\pi^2} - \frac{1}{12} + \mathcal O (\eps \pi/\alpha) \right) \right\} = \frac{1}{12} - \frac{\pi^2}{12 \alpha^2}. \]
Consequently, when $\alpha = \frac \pi j$ for some $j \in \N$,
\[ -  \sum_{k=1} ^{\kmax+|k_{min}|}  \frac{\gamma_e + \log |\sin(k \alpha)|}{4\pi \sin^2 (k \alpha)} = \frac{\gamma_e}{\pi} S  -  \sum_{k=1} ^{\kmax+|k_{min}|} \frac{\log |\sin(k \alpha)|}{4\pi \sin^2 (k \alpha)}\]
\[ = \frac{\gamma_e}{12 \pi} \left( 1 - \frac{\pi^2}{\alpha^2} \right)  -  \sum_{k=1} ^{\kmax+|k_{min}|} \frac{\log |\sin(k \alpha)|}{4\pi \sin^2 (k \alpha)}.\]
We recall that we obtained
\[  -  \sum_{k=1} ^{\kmax+|k_{min}|} \frac{\log |\sin(k \alpha)|}{4\pi \sin^2 (k \alpha)} = - \sum_{k=k_{min}} ^{-1}  \frac{ \log |\sin(k \alpha)|}{4\pi \sin^2 (k \alpha)} - \sum_{k=1} ^{k_{max}}  \frac{ \log |\sin(k \alpha)|}{4\pi \sin^2 (k \alpha)}.\]
When $\alpha = \frac \pi j$ for $j$ odd, $k_{min} = - k_{max}$, and we therefore have
\[- \sum_{k=k_{min}} ^{-1}  \frac{ \log |\sin(k \alpha)|}{4\pi \sin^2 (k \alpha)} - \sum_{k=1} ^{k_{max}}  \frac{ \log |\sin(k \alpha)|}{4\pi \sin^2 (k \alpha)} = - \sum_{k=1} ^{k_{max}}  \frac{ \log |\sin(k \alpha)|}{2\pi \sin^2 (k \alpha)}.\]
When $\alpha = \frac{\pi}{2n}$, we have $-n=k_{min}$ whereas $n-1 = k_{max}$.  In this case, however we note that
\[ \log| \sin(k_{min} \alpha)| = \log |\sin(\pi/2)| = 0.\]
Consequently, we may also in this case combine the sums, obtaining
\[- \sum_{k=k_{min}} ^{-1}  \frac{ \log |\sin(k \alpha)|}{4\pi \sin^2 (k \alpha)} - \sum_{k=1} ^{k_{max}}  \frac{ \log |\sin(k \alpha)|}{4\pi \sin^2 (k \alpha)} = - \sum_{k=1} ^{k_{max}}  \frac{ \log |\sin(k \alpha)|}{2\pi \sin^2 (k \alpha)} \] 
\[ =  - \sum_{k=1} ^{\ceil*{ \frac{\pi}{2\alpha} -1}} \frac{ \log |\sin(k \alpha)|}{2\pi \sin^2 (k \alpha)}. \]
\end{proof}

Based on the preceding proposition, we therefore have that if $\alpha \in (0, \pi)$ is not of the form $\frac{\pi}{2n}$ for any $n \in \N$, the variational formula in Theorem \ref{th:correctformula} simplifies to
\[ \frac{d}{d\alpha} (- \log \det (\Delta_{S_\alpha})) = \frac{1}{3\pi} + \frac{\pi}{12 \alpha^2}  -  \sum_{k=1} ^{\ceil*{ \frac{\pi}{2\alpha} -1} }  \frac{\gamma_e + \log |\sin(k \alpha)|}{2\pi \sin^2 (k \alpha)} \]
\[ + \frac{1}{\alpha} \sin(\pi^2/\alpha) \int_\R \frac{ - \log 2 + 2 \gamma_e + \log(1+\cosh(s))}{8 \pi (1+\cosh(s)) (\cosh(\pi s/\alpha) - \cos(\pi^2/\alpha))} ds.\]
In case $\alpha = \frac \pi j$ for some $j \in \N$, the expression simplifies to
\[\frac{d}{d\alpha} (- \log \det (\Delta_{S_\alpha})) = \frac{1}{3\pi} + \frac{\pi}{12 \alpha^2} - \frac{\gamma_e}{12\pi} \left( \frac{\pi^2}{\alpha^2} - 1 \right) - \frac{1}{2\pi} \sum_{k=1} ^{\ceil*{ \frac{\pi}{2\alpha} -1}} \frac{ \log | \sin(k \alpha)|}{\sin^2(k \alpha)}. \]
\end{proof} 

\begin{proof}[Proof of Theorem \ref{cthm:detcone}] 
The first displayed equation in this theorem, for $-\log(\det(\Delta_{C_\alpha}))$ is an immediate consequence of \eqref{eq:zeta_c2pi_a} together with a representation of the Barnes zeta function as given, for example, in \cite[p.~30]{kalvin2}; see also \cite{AuSa}. The variational formula is then obtained by differentiating with respect to the parameter $\alpha$, which is valid since the integrals converge absolutely.  To be more rigorous one may use the dominated convergence theorem.  The remaining variational formulas follow from the relationship between the zeta function on the cone and on the sector, together with the formulas obtained for the variation of the sector in Theorem \ref{t:sectors}. 
\end{proof}

\begin{appendix}
\section{Heat trace contribution for conical singularities} \label{s:heatcalcapp}
The heat kernel for a cone of angle $2 \alpha$ was obtained by Carslaw \cite{carslaw}
\begin{equation} H_{2\alpha} (r, \phi, r', \phi', t) = \frac{e^{-(r^2 + r'^2)/4t}}{8 \pi \alpha t} \int_{A_\phi} e^{r r' \cos(z-\phi)/2t} \frac{e^{ i \pi z/\alpha}}{e^{i  \pi z/\alpha} - e^{i  \pi \phi'/\alpha}} dz.\end{equation}
The contour $A_\phi$ in the complex plane is described and depicted in \cite[\S 6]{AldRow}; we omit this detail here as it is not pertinent to our present aim. This expression can be used to obtain the Dirichlet and Neumann heat kernels for a sector of angle $\alpha$ by the method of images as done in \cite[\S 6]{AldRow}.  One may also reverse the process to obtain the heat kernel for the cone of angle $2 \alpha$,
\begin{equation} \label{hk2alpha} H_{2\alpha} (r, \phi, r', \phi', t) = \frac 1 2 \left( \widetilde H_D + \widetilde H_N \right). \end{equation}
The meaning of $\widetilde H_D$ is that it is the odd extension of the Dirichlet heat kernel $H_D$, originally defined on the sector of angle $\alpha$, extended to an odd function on the double of the sector.  This is an odd function with respect to the involution that swaps the two identical sectors of angle $\alpha$.  Similarly, $\widetilde H_N$ is the even extension of the Neumann heat kernel $H_N$, originally defined on the sector of angle $\alpha$ and extended to the double of the sector, that of angle $2\alpha$, to be an even function with respect to this involution.

The heat kernel for an infinite sector of opening angle $\alpha$ with the Dirichlet and Neumann boundary conditions also admit an explicit integral representation.  Let
\[ A := \int_{0}^{\infty}K_{i\mu}(r \sqrt s)K_{i\mu}(r_0 \sqrt s) \cosh(\pi-|\phi_0-\phi|)\mu d\mu, \]
\[ B := \int_{0}^{\infty}K_{i\mu}(r \sqrt s)K_{i\mu}(r_0 \sqrt s) \frac{\sinh\pi\mu}{\sinh\alpha\mu}\cosh(\phi+\phi_0-\alpha)\mu d\mu, \]
\[ C :=  \int_{0}^{\infty}K_{i\mu}(r \sqrt s)K_{i\mu}(r_0 \sqrt s) \frac{\sinh(\pi-\alpha)\mu}{\sinh\alpha\mu}\cosh(\phi-\phi_0)\mu  d\mu. \]
Above $K_{i\mu}$ is the modified Bessel function of the second type.  The Dirichlet and Neumann Green's functions for the infinite sector of angle $\alpha$, are, respectively,
\[ G_D = \frac{1}{\pi^2} \left( A - B + C \right), \quad G_N = \frac{1}{\pi^2} \left( A + B + C \right).\]
The spectral parameter is $s$.  Using functional calculus one may prove that the heat kernel is obtained by taking the inverse Laplace transform in the $s$ variable, that is
\[ H_D(r, \phi, r', \phi', t) = \cL^{-1} (G_D)(t), \quad H_N(r, \phi, r', \phi', t) = \cL^{-1} (G_N)(t). \]
The contributions of the terms $A$, $B$, and $C$ to the heat trace were computed in \cite[\S 3]{nrs1}; see also \cite{vdbs}.  In particular, we computed the integral of each of these expressions, along the diagonal $r=r_0$ and $\phi=\phi_0$ over the region $[0, R]_r \times [0, \alpha]_\phi$ with respect to polar coordinates $(r, \phi)$.

The contribution from the $A$ term according to \cite[\S 3.1.1]{nrs} is
\begin{equation} \label{A-tr} \int_0 ^\alpha \int_0 ^R \frac{1}{\pi^2} \cL^{-1} (A) r dr d\phi = \frac{\alpha R^2}{8 \pi t} = \frac{ \textrm{Area}}{4 \pi t}. \end{equation}
The $B$ term contributes \cite[\S 3.1.2]{nrs1}
\begin{equation} \label{Btrace} \frac{R}{4 \sqrt{\pi t}} + O(\sqrt{t}), \quad t \downarrow 0. \end{equation}
Here we recognize the familiar perimeter term:
\[ \frac{ \textrm{ Perimeter }}{8 \sqrt{\pi t}}. \]
The term $C$ contributes to the trace \cite[\S 3.1.3]{nrs1}
\begin{equation} \label{Ctrace} \frac{\pi^2 - \alpha^2}{24\pi \alpha}.  \end{equation}

This shows that if we use the expression in \eqref{hk2alpha}, since we integrate over two copies of the sector of angle $\alpha$, the contribution to the $t^0$ term in the heat trace for the cone of angle $2\alpha$ is
\[ \frac 1 2 \left( 2 \frac{\pi^2 - \alpha^2}{24\pi \alpha} + 2 \frac{\pi^2 - \alpha^2}{24\pi \alpha} \right) =  \frac{\pi^2 - \alpha^2}{12 \pi \alpha}.\]
In terms of the opening angle of the cone,
\[ \gamma = 2 \alpha, \]
we may re-write this as
\[ \frac{\pi^2 - (\gamma/2)^2}{12 \pi \gamma/2} = \frac{(2\pi)^2 - \gamma^2}{24 \pi \gamma}. \]
Note that our calculation agrees with Cheeger's expression \cite[4.42]{cheeger}, and we have computed by a completely independent method.

\section{Polyakov formula for changing the angle of a finite cone}  \label{ss:vpcone}
An alternative way to obtain the differentiated Polyakov formula on a finite cone is by considering conformal transformations of the metric that represent a change in the cone angle, analogous to the approach of \cite{AldRow}.  In this case, the conformal factors have logarithmic singularities at the conical singularity. We represent a finite circular cone with opening angle $\alpha$ by $(M_\alpha,g_\alpha)$, where
\[ M_\alpha = [0,1]\times S^1_\alpha, \quad \quad S^1_\alpha =\{e^{i \frac{2\pi}{\alpha} \phi} \ | \ 0\leq \phi \leq \alpha\} \]
with metric
\[ g_\alpha = dr^2 + r^2 d\phi^2, \quad \quad r\in [0,1], \quad \phi \in S^1_\alpha. \]
With this description, we have that
\[ \int_{S^1_\alpha} d\phi = \alpha, \]
is the cone angle. Above we wrote $S^1_{\alpha}$ however, if there is no place to confusion, we drop the angle in the notation.

As in \cite[\S 2]{AldRow}, we fix a cone $Q=M_\beta$ with opening angle $\beta \leq \alpha$.  We use $(\rho, \theta)$ to denote polar coordinates on $Q$, so that the metric on $Q$ is
\[ g _\beta = d\rho^2 + \rho^2 d\theta^2,  \quad \quad \rho\in [0,1], \quad \theta \in S^1_\beta. \]
Analogously $Q$ has cone angle
\[\int_{S^1_\beta} d\theta = \beta.\]

In order to express the variation of the cone angle as a conformal transformation, we define the map
\[ \Psi_\gamma: Q \to M_\gamma, \quad (\rho, \theta) \mapsto \left( \rho^{\gamma/\beta},  \frac{\gamma \theta}{\beta} \right) = (r, \phi). \]
We consider the opening angle, $\gamma$ as a variable and will differentiate with respect to this variable and then evaluate at $\gamma = \alpha$.  The pull-back metric of the metric $g_\gamma$ is
\[ h_\gamma := \Psi^*_\gamma g_\gamma = e^{ 2 \sigma_{\gamma}} \left( d\rho^2 + \rho^2 d\theta^2 \right), \]
where
\[ \sigma_\gamma(\rho, \theta)  = \log \left( \frac \gamma \beta \right) + \left( \frac \gamma \beta - 1 \right) \log \rho. \]
The calculations to compute the variational Polyakov formula in this case are completely ana\-lo\-gous to those for the sector in \cite{AldRow}.  First we obtained the existence of a short time asymptotic expansion
\[ \Tr_{L^{2}(M_\alpha,g_\alpha)} \left( \cM_{(1+\log r)} e^{-t \Delta_\alpha}\right) \sim a_0 t^{-1} + a_1 t^{-\frac 1 2} + a_{2,0} \log(t) + a_{2,1} + \mathcal O(\sqrt t \log t), \]
for $t \downarrow 0$.  For a truncated cone with Dirichlet boundary condition at the smooth boundary component, $\ker(\Delta_{\alpha}) =\{0\}$,
therefore, the proof for finite sectors in \cite{AldRow} immediately implies the formula
\[ \left . \frac{d}{d \gamma} \zeta_{\Delta_{\gamma}} ' (0)\right|_{\gamma = \alpha} = \frac 2 \alpha \left( - \gamma_e a_{2,0} + a_{2,1} \right). \]
At this point we could continue with the proof as in the case for sectors in \cite{AldRow}, using Carslaw's formula for the heat kernel, now on an infinite cone, and look for all the terms and contributions to the formula given above. However, a more straightforward approach is obtained if we use the relationship between the determinant on the sector and the cone 
\[ \frac{d}{d \alpha}  \log \det (\Delta_{C_\alpha}) =  \frac{d}{d \alpha} \log \det (\Delta_{S_{\alpha/2}}) .\]
Then one simply uses the formula we have obtained on the right side.

\end{appendix}
\providecommand{\bysame}{\leavevmode\hbox to3em{\hrulefill}\thinspace}
\providecommand{\MR}{\relax\ifhmode\unskip\space\fi MR }
\providecommand{\MRhref}[2]{%
  \href{http://www.ams.org/mathscinet-getitem?mr=#1}{#2}
}
\providecommand{\href}[2]{#2}

\address{Departamento de Matem\'aticas y Estad\'{\i}stica. Universidad del Norte. Km 5 Via Puerto Colombia.  Area Metropolitana de Barranquilla, Colombia}
\email{claldana@uninorte.edu.co, clara.aldana@posteo.net}

\address{GCAP-CASPER, Department of Mathematics, Baylor University Waco, Texas, 76798, USA and Mathematical Reviews, American Mathematical Society, 416 4th Street, Ann Arbor, MI 48103, USA}
\email{Klaus\_Kirsten@baylor.edu}

\address{Department of Mathematics, Chalmers University of Technology and the University of Gothenburg, 41296 Gothenburg, Sweden}
\email{julie.rowlett@chalmers.se}

\end{document}